\documentclass[11pt,a4paper]{article}
\usepackage{amssymb,amsmath,amsthm,amsfonts,xcolor,paralist}
\usepackage{tikz}
\usetikzlibrary{patterns}
\usetikzlibrary{shapes}
\usepackage{fullpage}
\usepackage[unicode]{hyperref}
\usepackage{thmtools}
\usepackage{verbatim}
\usepackage{setspace}
\newtheorem{theorem}{Theorem}[section]
\newtheorem{lemma}[theorem]{Lemma}
\newtheorem{obs}[theorem]{Observation}

\newtheorem{ques}[theorem]{Question}

\newtheorem{claim}[theorem]{Claim}
\newtheorem{prop}[theorem]{Proposition}
\theoremstyle{definition}
\newtheorem{dfn}[theorem]{Definition}
\newtheorem{rmk}[theorem]{Remark}
\newcommand\ex{\ensuremath{\mathrm{ex}}}

\newcommand{\eps}{\varepsilon}
\newcommand{\om}{\omega}

\newcommand{\al}{\alpha}
\newcommand{\de}{\delta}

\newcommand{\vp}{\varphi}

\newcommand{\mc}[1]{\mathcal{#1}}
\newcommand{\fr}[1]{\mathfrak{#1}}

\newcommand{\bN}{\ensuremath{\mathbb{N}}}
\newcommand{\bP}{\ensuremath{\mathbb{P}}}
\newcommand{\bE}{\ensuremath{\mathbb{E}}}

\newcommand{\bR}{\ensuremath{\mathbb{R}}}

\newcommand{\card}[1]{\left| #1 \right|}
\newcommand{\floor}[1]{\left \lfloor #1 \right \rfloor}
\newcommand{\weight}{w_{k}}
\newcommand{\ceil}[1]{\left \lceil #1 \right \rceil}
\newcommand{\rt}{\right}
\newcommand{\lt}{\left}

\title{Colouring set families without monochromatic $k$-chains}
\author{
Shagnik Das\thanks{Freie Universit\"at Berlin, Institut f\"ur Mathematik, Arnimallee 3, 14195 Berlin, Germany. {\tt{shagnik@mi.fu-berlin.de}}. Research supported by GIF grant G-1347-304.6/2016.}
\and
Roman Glebov\thanks{School of Computer Science and Engineering, Hebrew University, Jerusalem 9190401, Israel. {\tt{roman.l.glebov@gmail.com}}. Research supported by the ERC grant 339096 `High-dimensional combinatorics' at the Hebrew University.}
\and
Benny Sudakov\thanks{Department of Mathematics, ETH, 8092 Zurich, Switzerland. {\tt benjamin.sudakov@math.ethz.ch}. Research supported in part by SNSF grant 200021-175573.}
\and
Tuan Tran\thanks{Department of Mathematics, ETH, 8092 Zurich, Switzerland. {\tt manh.tran@math.ethz.ch}. Research supported by the Alexander Humboldt Foundation.
}
}

\begin{document}
\maketitle
\begin{abstract}
A coloured version of classic extremal problems dates back to Erd\H{o}s and Rothschild, who in 1974 asked which $n$-vertex graph has the maximum number of 2-edge-colourings without monochromatic triangles. They conjectured that the answer is simply given by the largest triangle-free graph. Since then, this new class of coloured extremal problems has been extensively studied by various researchers. In this paper we pursue the Erd\H{o}s--Rothschild versions of Sperner's Theorem, the classic result in extremal set theory on the size of the largest antichain in the Boolean lattice, and Erd\H{o}s' extension to $k$-chain-free families.

Given a family $\mc F$ of subsets of $[n]$, we define an \emph{$(r,k)$-colouring} of $\mc F$ to be an $r$-colouring of the sets without any monochromatic $k$-chains $F_1 \subset F_2 \subset \hdots \subset F_k$.  We prove that for $n$ sufficiently large in terms of $k$, the largest $k$-chain-free families also maximise the number of $(2,k)$-colourings.  We also show that the middle level, $\binom{[n]}{\floor{n/2}}$, maximises the number of $(3,2)$-colourings, and give asymptotic results on the maximum possible number of $(r,k)$-colourings whenever $r(k-1)$ is divisible by three.

\end{abstract}

\section{Introduction}

Sperner's Theorem on the size of the largest antichain in the Boolean lattice, which dates back to 1928, is one of the fundamental theorems in extremal set theory.  An antichain is a family of sets $\mc F \subseteq 2^{[n]}$ where no set is contained in another; that is, $F_1 \not\subset F_2$ for all distinct $F_1, F_2 \in \mc F$.  Sperner~\cite{Spe28} proved that $2^{[n]}$ does not have any antichains with more than $\binom{n}{\floor{n/2}}$ sets, a bound that is easily seen to be tight by considering the family of all sets of size $\floor{n/2}$.

Since its inception, Sperner's Theorem has inspired a great deal of further research, as various extensions have been proven. For instance, Erd\H{o}s~\cite{Erd45} determined the largest family $\mc F \subseteq 2^{[n]}$ without a $k$-chain; that is, without sets $F_1 \subset F_2 \subset \hdots \subset F_k$.  Sperner's Theorem corresponds to the case $k = 2$, and Erd\H{o}s extended this by showing one ought to take the $k-1$ largest uniform levels in $2^{[n]}$.  When $n-k$ is odd, there are two isomorphic extremal families, and we shall not distinguish between them in this paper.

In this paper we solve some instances of the Erd\H{o}s--Rothschild extensions of these theorems.  Before presenting our new results, we first introduce the class of Erd\H{o}s--Rothschild problems, which are multicoloured versions of classic extremal problems.

\subsection{Erd\H{o}s--Rothschild problems}

One of the earliest results in extremal combinatorics was obtained in 1907, when Mantel~\cite{Man07} showed that any $n$-vertex graph with more than $\floor{n^2 / 4}$ edges must contain a triangle, a bound that the complete balanced bipartite graph shows to be best possible.  This was later extended by Tur\'an~\cite{Tur41}, who determined the largest number of edges an $n$-vertex $K_t$-free graph can have, for any $t$.  More generally, the Tur\'an number of a graph $H$, denoted $\ex(n,H)$, is the maximum number of edges in an $H$-free graph on $n$ vertices.

Various extensions and variations of the Tur\'an problem have been pursued through the years.  One such problem was proposed by Erd\H{o}s and Rothschild~\cite{Erd74} in 1974.  They asked which $n$-vertex graph had the maximum number of two-edge-colourings without monochromatic triangles.  Clearly, since the complete balanced bipartite graph is itself triangle-free, any two-colouring of its edges is monochromatic-triangle-free, and hence one can have at least $2^{\floor{n^2/4}}$ such colourings.

In order to do better, one would have to take a graph with more edges, which, by Mantel's Theorem, implies the existence of triangles.  These triangles impose restrictions on the two-colourings, as their edges cannot be coloured monochromatically.  Erd\H{o}s and Rothschild believed that the restrictions from the triangles would more than counteract the extra possibilities offered by the additional edges, and conjectured in~\cite{Erd74} that $2^{\floor{n^2/4}}$ is in fact the best one can do.  It was some twenty years before Yuster~\cite{Yus96} proved the conjecture for $n\ge 6$.

There are some obvious generalisations of this problem of Erd\H{o}s and Rothschild --- one may ask it for graphs other than the triangle, and one may increase the number of colours used.  Let an $(r,F)$-colouring of a graph $G$ be an $r$-colouring of its edges without any monochromatic copies of $F$.  The question is then to determine which $n$-vertex graphs $G$ maximise the number of $(r,F)$-colourings.

Once again, a natural lower bound is obtained by considering all $r$-edge colourings of the largest $F$-free graph, which shows that the maximum number of $(r,F)$-colourings of an $n$-vertex graph is at least $r^{\ex(n,F)}$.  In 2004, Alon, Balogh, Keevash and Sudakov~\cite{ABKS04} greatly extended Yuster's result by showing this lower bound was tight whenever 
$F = K_t$ for $t \ge 3$, $r \in \{2, 3\}$ and $n\ge n_0(t,r)$ sufficiently large (see \cite{HJ18} for an improved bound on $n_0(t,r)$).
Interestingly, they further demonstrated that this was not the case when one has four or more colours, providing some better bounds in this range.  Exact results were later obtained by Pikhurko and Yilma~\cite{PY12}, who determined which $n$-vertex graphs maximise the number of $(4,K_3)$- and $(4,K_4)$-colourings.  Pikhurko, Staden and Yilma~\cite{PSY17} have recently introduced an asymmetric version of this problem, reducing its log-asymptotic solution to a large but finite optimisation problem.

In recent years, there have been a variety of papers studying the Erd\H{o}s--Rothschild problem in various settings.  Lefmann, Person, R\"odl and Schacht~\cite{LPRS09} studied the problem in the setting of three-uniform hypergraphs, with monochromatic copies of the Fano plane forbidden, before Lefmann, Person and Schacht~\cite{LPS10} considered arbitrary $k$-uniform hypergraphs.  Further results along this line of research can be found in~\cite{HKL12a, HKL14, HKL15, LP13}.  Moving the problem into the domain of extremal set theory, Hoppen, Kohayakawa and Lefmann~\cite{HKL12b} solved the Erd\H{o}s--Rothschild extension of the famous Erd\H{o}s--Ko--Rado Theorem~\cite{EKR61}.  Hoppen, Lefmann and Odermann~\cite{HLO16} provided initial results for the vector space analogue of the Erd\H{o}s--Ko--Rado Theorem, before Clemens, Das and Tran~\cite{CDT16} presented a unified proof extending some of these results. In the context of additive combinatorics, the Erd\H{o}s--Rothschild extension for sum-free sets has been pursued by H\`an and Jim\'enez~\cite{HJ17} for abelian groups and Liu, Sharifzadeh and Staden~\cite{LSS17} for subsets of the integers.

\subsection{Our results}

Following their work on intersecting vector spaces, Hoppen, Lefmann and Odermann~\cite{HLO16} suggested the investigation of Erd\H{o}s--Rothschild problems in the context of the power set lattice.  Sperner's Theorem on antichains is arguably the most important extremal result in this setting, and we consider the corresponding Erd\H{o}s--Rothschild extension, as well as that of Erd\H{o}s's result on $k$-chain-free families.

To this end, given a set family $\mc F \subseteq 2^{[n]}$, we define an \emph{$(r,k)$-colouring} of $\mc F$ to be an $r$-colouring of the sets in $\mc F$ without any monochromatic $k$-chains $F_1 \subset F_2 \subset \hdots \subset F_k$.  Given $r \ge 2$ and $n \ge k \ge 2$, the goal is to determine which families have the maximum possible number of $(r,k)$-colourings.  Let us define $f(r,k;n)$ to be this maximum.

Let us first consider the case $k = 2$.  Here we forbid a monochromatic $2$-chain or, in other words, require that each colour class be an antichain.  As before, any $r$-colouring of an antichain is an $(r,2)$-colouring.  By Sperner's Theorem~\cite{Spe28}, the largest antichain in $2^{[n]}$ has size $\binom{n}{\floor{n/2}}$.  The lower bound
\begin{equation} \label{ineq:2chainLB} 
	f(r,2; n) \ge r^{\binom{n}{\floor{n/2}}}
\end{equation}
thus follows, and we seek to determine if this is best possible.

This is somewhat trivial when $r = 2$.  Indeed, let $\mc F \subseteq 2^{[n]}$ be a family maximising the number of $(2,2)$-colourings, and let $\mc F_0 \subseteq \mc F$ be the subfamily of all minimal sets in $\mc F$.  Observe that any $(2,2)$-colouring of $\mc F$ is determined by the colouring of $\mc F_0$, since any set $F \in \mc F \setminus \mc F_0$ contains a set in $\mc F_0$, and must thus be oppositely-coloured.  Hence $\mc F$ can have at most $2^{\card{\mc F_0}}$ $(2,2)$-colourings, and since $\mc F_0$ is an antichain, we have $\card{\mc F_0} \le \binom{n}{\floor{n/2}}$.  Thus~\eqref{ineq:2chainLB} is tight for $r = 2$, and it is not hard to show that one only has equality for the largest antichains; that is, when $\mc F$ is (one of) the middle layer(s) of the Boolean lattice.

To show that the lower bound remains tight for $r= 3$ requires considerably more work, and is the content of our first theorem.

\begin{restatable}{theorem}{threecolours}
\label{thm:3-colours}
There is some $n_0 \in \mathbb{N}$ such that for every integer $n \ge n_0$, the number of $(3,2)$-colourings of a family $\mc F \subseteq 2^{[n]}$ is at most $3^{\binom{n}{\floor{n/2}}}$.  Moreover, we have equality if and only if $\mc F = \binom{[n]}{\floor{n/2}}$ or $\mc F = \binom{[n]}{\ceil{n/2}}$.
\end{restatable}

Just as for the Erd\H{o}s--Rothschild extension for Tur\'an's Theorem, we cannot expect equality in~\eqref{ineq:2chainLB} to hold for larger $r$.  Indeed, for $r = 4$, suppose $n = 2m + 1$, and consider the family $\mc F = \binom{[n]}{m} \cup \binom{[n]}{m + 1}$.  Partition the four colours into two pairs, and colour the sets in $\binom{[n]}{m}$ with one of the pairs, and those in $\binom{[n]}{m+1}$ with the other.  Each such colouring, of which there are $2^{\binom{n}{m}} \cdot 2^{\binom{n}{m+1}} = 4^{\binom{n}{\floor{n/2}}}$, is clearly a $(4,2)$-colouring.  This matches the lower bound of~\eqref{ineq:2chainLB}.  However, as there are six ways to partition the colours, and these give rise to almost disjoint sets of $(4,2)$-colourings, we find that $\mc F$ has more $(4,2)$-colourings than the largest antichain $\binom{[n]}{\floor{n/2}}$.  For larger values of $r$, this family $\mc F$ has exponentially more $(r,2)$-colourings than the lower bound given above.

We now turn to the case when $k \ge 3$.  Recall that Erd\H{o}s~\cite{Erd45} proved that the largest $k$-chain-free family in $2^{[n]}$ consists of the $k-1$ largest uniform levels; that is, the collection of all sets whose sizes lie between $\floor{(n-k+2)/2}$ and $\floor{(n+k-2)/2}$.  Let the size of this family be denoted by $m_{k-1}$, so that $m_{k-1} = \sum_{i = \floor{(n-k+2)/2}}^{\floor{(n+k-2)/2}} \binom{n}{i}$.  Since every $r$-colouring of this family is an $(r,k)$-colouring, we have the inequality
\begin{equation} \label{ineq:kchainLB}
	f(r,k;n) \ge r^{m_{k-1}}.
\end{equation}

The case $k \ge 3$ appears to be rather more complicated than $k = 2$.  Indeed, it was trivial to bound the number of $(2,2)$-colourings that a family could have.  When $r = 2$ and $k \ge 3$, the lower bound of~\eqref{ineq:kchainLB} is again sharp, but our proof is much more involved, and forms our next result.

\begin{restatable}{theorem}{twocolours}
\label{thm:2-colours}
There exists an absolute constant $C>0$ such that for integers $n$ and $k$ with $k\ge 2$ and $n\ge Ck^4\log k$, the number of $(2,k)$-colourings of a family $\mc F \subseteq 2^{[n]}$ is at most $2^{m_{k-1}}$.  Moreover, we have equality if and only if $\mc F$ is a $k$-chain-free family of maximum size.
\end{restatable}

The common strategy for proving Theorems \ref{thm:3-colours} and \ref{thm:2-colours} is to decompose any family
$\mc F$ with at least one $(r,k)$-colouring into a family $\mc F'$ with few $k$-chains, and a family $\mc F''$ whose elements are contained in many $k$-chains.\footnote{The actual proofs use much more complex decompositions.} Those $k$-chains that contain at least one member from $\mc F''$ impose many restrictions on our colourings. On the other hand, it follows from the supersaturation results that the size of $\mc F'$ is at most $(k-1+o(1))\binom{n}{\floor{n/2}}$. Piecing this information together allows us to show that if $\mc F$ has the maximum number of $(r,k)$-colourings, then $\mc F''=\emptyset$ and $\mc F=\mc F'$ is a union of $k-1$ layers of the Boolean lattice.

For larger values of $r$, it is fairly straightforward to see that the lower bound in~\eqref{ineq:kchainLB} is not best possible; indeed, for $r \ge 5$, one can again do exponentially better.  Our final result uses the machinery of hypergraph containers to provide upper bounds on $f(r,k;n)$. These determine $f(r,k;n)$ log-asymptotically for more than half of pairs $(r,k)$.

\begin{restatable}{prop}{logasym}
\label{prop:logasym}
For $k \ge 2$, $r \ge 3$ and $\eps > 0$ there exists $n_0(r,k,\eps) \in \bN$ such that for $n \ge n_0(r,k,\eps)$ we have
\[ f(r,k;n) \le 3^{\frac13 r (k - 1 + \eps) \binom{n}{\floor{n/2}}}. \]
Moreover, if $r(k-1)$ is divisible by three, then
\[ f(r,k;n) \ge  3^{\frac13 r (k - 1 - \eps) \binom{n}{\floor{n/2}}} \]
as well.
\end{restatable}

\subsection{Organisation and notation}\label{subsec:orgnot}

The remainder of this paper is organised as follows.  In Section~\ref{sec:3-colours} we prove Theorem~\ref{thm:3-colours}, and then prove Theorem~\ref{thm:2-colours} in Section~\ref{sec:2-colours}.  Proposition~\ref{prop:logasym} is proved in Section~\ref{sec:logasym}.  We remark that these sections are independent of one another, and can be read in any order.  We close the paper with some concluding remarks in Section~\ref{sec:conclusion}.

One notion we shall use throughout is that of a \emph{comparability graph}.  Given a family $\mc F$ of subsets of $[n]$, the comparability graph of $\mc F$ is the graph $G(\mc F)$ whose vertices are the sets in $\mc F$, with an edge $\{F_1,F_2\}$ whenever $F_1\subsetneq F_2$.  For any set $F \subseteq [n]$, the \emph{up-degree} (respectively \emph{down-degree}) of $F$ in $\mc F$, denoted by $d^{+}(F,\mc F)$ (respectively $d^{-}(F,\mc F)$), is the number of sets $F'$ in $\mc F$ such that $F \subsetneq F'$ (respectively $F' \subsetneq F$). The \emph{degree} of $F$ in $\mc F$, denoted by $d(F,\mc F)$, is the sum of its up-degree and down-degree in $\mc F$, and counts the sets in $\mc F$ comparable to $F$ (excluding $F$ itself, if $F \in \mc F$).  We denote by $N^+(F ,\mc F)$, $N^-(F, \mc F)$ and $N(F, \mc F)$ the subfamilies of up-, down- and all neighbours of a set $F$ in the family $\mc F$.

That apart, we make use of standard combinatorial notation.  We denote by $[n]$ the set $\{ 1, 2, \hdots, n\}$, and shall take that to be the ground set for our set families.  Given a set $X$ and some $k \in \mathbb{N}$, $\binom{X}{k}$ is the family of all $k$-subsets of $X$, while $2^X$ is the family of all subsets of $X$, regardless of size.  Finally, unless stated otherwise, all logarithms are binary.

\section{Three-colourings without monochromatic two-chains}\label{sec:3-colours}

In this section, we will show that the number of $(3,2)$-colourings of a set family is maximised only by the largest antichains; that is, by the middle levels, $\binom{[n]}{\floor{n/2}}$ or $\binom{[n]}{\ceil{n/2}}$.

\threecolours*

Before delving into the details of the proof, we provide a heuristic argument for why the largest antichains should also maximise the number of $(3,2)$-colourings.  In an antichain, a $(3,2)$-colouring can be formed by arbitrarily assigning each of the sets one of the three colours.  Therefore, in order to have more $(3,2)$-colourings, one must consider a larger family.

By Sperner's Theorem, any larger family must contain comparable pairs.  These pairs place restrictions on our colourings: once the colour of a given set is fixed, any comparable sets have at most two of the three colours available to them. To overcome these restrictions, the family must in fact be significantly larger than the largest antichains.

However, one can then show that the family must contain many sets that are comparable to a large number of other sets. This in turn leads to even stronger restrictions on the $(3,2)$-colourings of the family.  Indeed, consider such a set $F$, and let $\mc G$ be the sets in the family to which $F$ is comparable.  In a typical colouring, we would expect to see at least two colours appearing in $\mc G$, in which case the colour of $F$ is determined, rendering this set redundant.  In order for $F$ to actually increase the number of $(3,2)$-colourings, we would need $\mc G$ to be monochromatic, which is a very restrictive condition.

It thus appears unlikely that larger families could have more $(3,2)$-colourings.  In what follows, we formalise this argument and show this intuition to be true.

\begin{proof}

Let $\mc F \subseteq 2^{[n]}$ maximise the number of $(3,2)$-colourings and let $\fr C$ be the set of all $(3,2)$-colourings of $\mc F$.  The lower bound of~\eqref{ineq:2chainLB} implies
\begin{equation}\label{good-colourings_lower-I}
\card{\fr C} \ge 3^{\binom{n}{\floor{n/2}}}.
\end{equation}

We must establish a matching upper bound on $\card{\fr C}$. We begin by deducing some structural information about $\mc F$.  Suppose $\mc F$ contains a $3$-chain $F_1\subset F_2\subset F_3$.  If any two of $F_1, F_2$ and $F_3$ shared the same colour, they would form a monochromatic $2$-chain.  Hence in every colouring in $\fr C$, $F_1$, $F_2$ and $F_3$ must receive pairwise different colours, and thus the colour of $F_3$ is determined by those of $F_1$ and $F_2$.  It follows that $\mc F \setminus \{ F_3 \}$ has at least as many $(3,2)$-colourings as $\mc F$.

We therefore may assume that $\mc F$ is $3$-chain-free, and can thus be partitioned into two antichains by Mirsky's theorem~\cite{Mir71}.\footnote{This does not affect our claims of uniqueness; if we started with a family $\mc F$ that was not the middle level of the Boolean lattice, the $3$-chain-free subfamily would also not be the middle level, as it would still have comparable pairs.} In other words, the comparability graph $G(\mc F)$ of $\mc F$, as defined in Section~\ref{subsec:orgnot}, is bipartite.  We denote by $\mc I$ the subfamily of $\mc F$ consisting of isolated vertices of $G(\mc F)$; that is, those sets in $\mc F$ that are incomparable to all other sets in $\mc F$.

Next, we iteratively remove from the remainder $\mc F\setminus \mc I$ sets of degree at most $\sqrt{n}$ in the induced subgraph of $G(\mc F)$, and let $\mc S$ be the subfamily of the sparse sets that have been removed.  We are left with a bipartite subgraph of minimum degree at least $\sqrt{n}$, and call the subfamilies corresponding to the two independent sets $\mc A$ and $\mc B$.  Hence $d(A,\mc B) \ge \sqrt{n}$ for every $A\in \mc A$, and $d(B,\mc A) \ge \sqrt{n}$ for every $B\in \mc B$.

Given this structure, our goal is to show that a relatively small number of sets in $\mc F$ are contained in enough comparable pairs to limit the number of $(3,2)$-colourings of $\mc F$.  To this end, we now define some notation we will use throughout this proof.

\begin{dfn}\label{def:three-color}
Fix a subfamily $\mc X\subseteq \mc B$ and a $(3,2)$-colouring $\vp \in \fr C$.  We write $\mc B_{\text{mc}}^{\vp}$ for the subfamily of $\mc B$ consisting of vertices whose neighbourhoods in $\mc A$ are monochromatic under $\vp$.  Let $\mc A_3=\{F\in\mc A:d(F,\mc X)=0\}$, $\mc A_{2,\vp}=\{F\in\mc A:d(F,\mc B_{\text{mc}}^{\vp})\ge \sqrt{n}, d(F,\mc X\cap \mc B_{\text{mc}}^{\vp})=0\}$ and $\mc A_{1,\vp}=\{F\in \mc A:d(F,\mc B_{\text{mc}}^{\vp})<\sqrt{n}\}$. Note that $\mc A_3$ is not necessarily disjoint from $\mc A_{2,\vp} \cup \mc A_{1,\vp}$.
\end{dfn}
The following observations follow straightforwardly from the above definition.  
\begin{obs}\label{obs:three-color}
The families $\mc B_{\text{mc}}^{\vp}$ and $\mc A_{1, \vp}$ are determined by the restriction of $\vp$ to $\mc A$. Moreover, $\mc A_3$ depends only on $\mc X$, while $\mc A_{2, \vp}$ depend only on $\mc B_{\text{mc}}^{\vp}$ and $\mc X$.  
\end{obs}
The following claim, to be proven at the end of this section, is the key of the proof.

\begin{claim}\label{claim:random-subset}
For sufficiently large $n$, there exist families $\mc X \subseteq \mc B$ and $\fr C'\subseteq \fr C$ such that
\begin{itemize}
	\item[\rm (i)] $\card{\mc X}\le 0.01\card{\mc B}$, 
	\item[\rm (ii)] $\card{\fr C'} \ge \frac{9}{10}\card{\fr C}$, and
	\item[\rm (iii)] $\max\{\card{\mc A_3},\card{\mc A_{2,\vp}}\}\le 0.01\card{\mc A}$ for all colourings $\vp$ in $\fr C'$. 
\end{itemize}
\end{claim}

Before continuing, let us remark that, by \eqref{good-colourings_lower-I} and Claim \ref{claim:random-subset} (ii), the size of $\fr C'$ is at least
\begin{equation}\label{good-colourings_lower-II}
\card{\fr C'} \ge \tfrac{9}{10}\card{\fr C}\ge\tfrac{9}{10}\cdot 3^{\binom{n}{\floor{n/2}}}.
\end{equation}

With $\mc X$ as in Claim~\ref{claim:random-subset}, we collect the following estimates, whose proofs are also deferred to the end of this section.

\begin{claim}\label{claim:sperner-type} 
For $n$ sufficiently large, the following inequalities hold:
\begin{itemize} 
\item[\rm(a)] $\max\{|\mc A|,\card{\mc B}\} \le \binom{n}{\floor{n/2}}-\card{\mc I}$,
\item[\rm(b)] $\card{\mc S}\le 1.01\left(\binom{n}{\floor{n/2}}-\card{\mc I}\right)$, and
\item[\rm(c)] $\card{\mc A_{2,\vp}\cup \mc A_{1,\vp}}+\card{\mc B_{\text{mc}}^{\vp}} \le 1.02\left(\binom{n}{\floor{n/2}}-\card{\mc I}\right)-\card{\mc S}$ for every $\vp \in \fr C'$.
\end{itemize}
\end{claim}

Assuming Claims~\ref{claim:random-subset} and~\ref{claim:sperner-type}, we shall provide the desired upper bound on $\card{\fr C}$.  Indeed, we first appeal to Claim~\ref{claim:random-subset} to obtain a subset of vertices $\mc X \subseteq\mc B$ and a subfamily of colourings $\fr C'\subseteq \fr C$ with properties (i)--(iii).

We now show that for an arbitrary colouring $\vp$ in $\fr C'$, it is possible to reveal $\vp$ gradually by asking a small number of questions with a bounded number of possible answers. This will show that, unless $\mc F$ is a middle level of the Boolean lattice, $\card{\fr C'}$ is smaller than the lower bound in~\eqref{good-colourings_lower-II}, thus proving that the middle levels are the only families with the maximum possible number of $(3,2)$-colourings.

First, for every $x\in \mc X$, we ask for $\vp(x)$, and also whether its neighbourhood in $\mc A$ is monochromatic under $\vp$ and, if it is, which colour it has. Because there are three options for $\vp(x)$ and two colours may be assigned to the monochromatic neighbourhood of $x$, the number of possible answers for each vertex $x\in \mc X$ is at most $3+3\cdot 2=9$. So by Claims~\ref{claim:random-subset} (i) and~\ref{claim:sperner-type} (a), it follows that the total number of answers at this stage is not greater than $9^{\card{\mc X}} \le 3^{0.02\card{\mc B}} \le 3^{0.02\left(\binom{n}{\floor{n/2}}-\card{\mc I}\right)}$.
Given the answers for the sets in $\mc X$, we shall bound the number of possibilities to finish the colouring of $\mc A \cup \mc B$. 

The answers give rise to a unique partition $\mc A=\mc A_1\cup\mc A_2\cup\mc A_3$, where $\mc A_1$ is the union of monochromatic neighbourhoods of vertices in $\mc X$, and $\mc A_3$ consists of those sets in $\mc A$ without a neighbour in $\mc X$ (see Definition~\ref{def:three-color} and Observation~\ref{obs:three-color}). Note that each set in $\mc A_1$ has its colour determined. Since any set with a neighbour in $\mc X$ has at most two colours left (as it cannot have the same colour as its neighbour), only those sets in $\mc A_3$ could still have three colours available. As $|\mc A_3| \le 0.01 \left(\binom{n}{\floor{n/2}}-|\mc I| \right)$ by Claims~\ref{claim:random-subset} (iii) and \ref{claim:sperner-type} (a), we see that there are at most $3^{\card{\mc A_3}}2^{\card{\mc A_2}} \le 3^{0.01 \left(\binom{n}{\floor{n/2}}-|\mc I| \right)}2^{|\mc A_2|}$
ways for the sets in $\mc A$ to be coloured.

After specifying the restriction of $\vp$ to $\mc A$, we can identify the families $\mc B_{\text{mc}}^{\vp}, \mc A_{2,\vp}$ and $\mc A_{1,\vp}$ (see Definition \ref{def:three-color} and Observation~\ref{obs:three-color}). Because $\mc B_{\text{mc}}^{\vp}$ consists of all sets in $\mc B$ whose neighbourhoods in $\mc A$ are monochromatic, each sets in $\mc B_{\text{mc}}^{\vp}$ can receive at most $2$ colours, while sets in $\mc B\setminus(\mc X\cup\mc B_{\text{mc}}^{\vp})$ only have one colour available. It follows that the number of possibilities of extending the colouring to $\mc B$ is at most $2^{\card{\mc B_{\text{mc}}^{\vp}}}$. 

Note that $\mc A_2\subseteq \mc A_{2,\vp}\cup\mc A_{1,\vp}$, because every element $F\in \mc A_2\setminus \mc A_{1,\vp}$ has at least $\sqrt{n}$ neighbours in $\mc B_{\text{mc}}^{\vp}$ (as $F \notin \mc A_{1,\vp}$) and no neighbours in $\mc X\cap \mc B_{\text{mc}}^{\vp}$ (otherwise the colour of $F$ is already determined). Claim~\ref{claim:sperner-type} (c) thus force $\card{\mc A_2}+\card{\mc B_{\text{mc}}^{\vp}}\le 1.02\left( \binom{n}{\floor{n/2}}-\card{\mc I}\right)-\card{\mc S}$. Therefore, given the answers for all $x\in \mc X$, the number of ways to colour $\mc A\cup\mc B$ is at most $3^{0.01 \left(\binom{n}{\floor{n/2}}-|\mc I| \right)}2^{|\mc A_2|}2^{\card{\mc B_{{\text{mc}}}^{\vp}}} \le \left(3^{0.01}2^{1.02}\right)^{ \binom{n}{\floor{n/2}}-\card{\mc I}}2^{-\card{\mc S}}$.

We proceed to bound the number of ways to extend a given colouring of $\mc A \cup \mc B$ to a colouring of $\mc F$. To this end, let $\{\{F_1,G_1\},\ldots,\{F_t,G_t\}\}$ be a maximal matching in the comparability graph $G(\mc S)$. Since $F_i$ and $G_i$ are comparable for each $1 \le i \le t$, there are at most $6$ ways to colour each pair $\{F_i,G_i\}$, and hence the number of possible colourings for the matching is at most $6^t$. On the other hand, it follows from the definition of $\mc I$ and the maximality of the matching that every set $F$ in $\mc S$ not in the matching is adjacent to some previously-coloured set in $\mc A\cup\mc B\cup\{F_1,G_1,\ldots,F_s,G_s\}$, and therefore has at most two available colours. The family $\mc S$ can therefore be coloured in at most $6^{t}2^{\card{\mc S}-2t}$ possible ways, which is at most $\sqrt{6}^{\card{\mc S}}$ since $t\le \card{\mc S}/2$. Finally, the number of ways to colour $\mc I$ is $3^{\card{\mc I}}$.

Putting these inequalities together, we can bound the number of colourings in $\fr C'$ as follows
\begin{align}\label{good-colourings_upper}
\notag \card{\fr C'} &\le 3^{0.02\left(\binom{n}{\floor{n/2}}-\card{\mc I}\right)}\cdot \left(3^{0.01}2^{1.02}\right)^{ \binom{n}{\floor{n/2}}-\card{\mc I}}2^{-\card{\mc S}}\cdot \sqrt{6}^{\card{\mc S}}3^{\card{\mc I}}\\
\notag & = 3^{\binom{n}{\floor{n/2}}}\left(3^{-0.97}\cdot 2^{1.02}\right)^{\binom{n}{\floor{n/2}}-\card{\mc I}}\left(\sqrt{6} / 2 \right)^{\card{\mc S}}\\
\notag & \le 3^{\binom{n}{\floor{n/2}}}\left(3^{-0.97}\cdot 2^{1.02}\right)^{\binom{n}{\floor{n/2}}-\card{\mc I}}\left(\sqrt{6} / 2 \right)^{1.01\left(\binom{n}{\floor{n/2}}-\card{\mc I}\right)}\\
&\le 3^{\binom{n}{\floor{n/2}}}0.86^{\binom{n}{\floor{n/2}}-\card{\mc I}},
\end{align}
where the second inequality follows from Claim~\ref{claim:sperner-type} (b), and the last holds since $\card{\mc I}\le \binom{n}{\floor{n/2}}$ due to Sperner's Theorem.

From \eqref{good-colourings_lower-II} and \eqref{good-colourings_upper}, we find that $\binom{n}{\floor{n/2}}-\card{\mc I}\le 0$.  Hence $\mc I$ must be an antichain of size $\binom{n}{\floor{n/2}}$, and therefore one of the middle levels of $2^{[n]}$. Since $\mc I$ is the set of isolated vertices in the comparability graph $G(\mc F)$, and any other set is comparable to some sets in the middle levels, this forces $\mc F=\mc I$, completing the proof.
\end{proof}

It remains to prove Claims~\ref{claim:random-subset} and~\ref{claim:sperner-type}, a task we now begin.

\begin{proof}[Proof of Claim \ref{claim:random-subset}]
If $\mc A=\emptyset$, then $\mc X=\emptyset$ and $\fr C= \fr C'$ trivially have the desired properties.  We thus assume $\mc A \ne \emptyset$. By the definition of $\mc A$ and $\mc B$, we must have $\card{\mc B}\ge \sqrt{n}$. Observe that for every colouring $\vp$, the set $\mc B_{\text{mc}}^{\vp}$ is well-defined and independent of the random choice of $\mc X$ that we shall now make. Let $\mc X\sim (\mc B)_p$ denote the random subfamily of $\mc B$, where each set in $\mc B$ is included independently with probability $p=1/\log n$. Assuming $n$ is sufficiently large and applying the Chernoff bound (see, e.g., \cite[Corollary A.1.14]{AS16}) we have 

\begin{equation}\label{chernoff-size}
\bP(\card{\mc X} > 0.01 \card{\mc B}) \le \exp\left(-n^{1/3}\right).
\end{equation}

For a $(3,2)$-colouring $\vp$ of $\mc F$, let $E_{\vp}$ be the event that $\card{\mc A_{2,\vp}\cup\mc A_3}>0.01\card{\mc A}$. We will show it is unlikely that this occurs for many colourings simultaneously.  More precisely,
\begin{equation}\label{bad-event}
\bP\left(\card{\{\vp\in \fr C: E_{\vp}\}} > \frac{1}{10}\card{\fr C}\right) \le 1000 \exp\left(- n^{1/3} \right).
\end{equation}
Clearly, \eqref{chernoff-size} and \eqref{bad-event} together imply the existence of $\mc X\subseteq \mc B$ and $\fr C'\subseteq \fr C$ with the desired properties. 

It thus remains to show that \eqref{bad-event} holds. As every element in $\mc A_3$ has at least $\sqrt{n}$ neighbours in $\mc B$, none of which are selected in the random subfamily $\mc X$, the union bound gives
\[
\bE[\card{\mc A_3}] \le \card{\mc A}(1-p)^{\sqrt{n}} \le \card{\mc A} \exp \left(-p\sqrt{n} \right).
\]
Similarly, since $\mc A_{2,\vp}$ has $\sqrt{n}$ neighbours in $\mc B_{\text{mc}}^{\vp}\setminus {\mc X}$, we find
\[
\bE[\card{\mc A_{2,\vp}}] \le \card{\mc A}(1-p)^{\sqrt{n}} \le \card{\mc A} \exp \left(-p\sqrt{n} \right).
\]
Combining these bounds with Markov's inequality, we obtain
\[
\bP(E_{\vp})=\bP(\card{\mc A_{2,\vp}\cup\mc A_3}>0.01\card{\mc A}) \le \frac{2\card{\mc A} \exp \left(-p\sqrt{n}\right)}{0.01\card{\mc A}}\le 100 \exp \left( - n^{1/3} \right).
\] 
Linearity of expectation then gives $\bE\left(\card{\{\vp \in \fr C: E_{\vp}\}}\right)=\sum_{\vp \in \fr C}\bP(E_{\vp}) \le 100 \exp \left(- n^{1/3} \right) \card{\fr C}$, and hence another application of Markov's inequality implies
\[
\bP\left(\card{\{\vp\in \fr C: E_{\vp}\}} > \frac{1}{10}\card{\fr C}\right)\le \frac{100 \exp \left(-n^{1/3} \right) \card{\fr C}}{\card{\fr C} / 10}=1000 \exp\left(- n^{1/3}\right).
\]
This finishes our proof of Claim~\ref{claim:random-subset}.
\end{proof}	

We conclude this section with a proof of Claim~\ref{claim:sperner-type}, for which we shall use the following special case of a result of Kleitman~\cite{Kle66}.  This provides a lower bound on the number of comparable pairs in a large set family of a given size (see also~\cite{DGS15}, which characterises the extremal families).

\begin{theorem}[Kleitman~\cite{Kle66}] \label{thm:Kleitman}
A subfamily of $2^{[n]}$ with $\binom{n}{\floor{n/2}}+t$ sets must contain at least $\ceil{\frac{n+1}{2}} t$ comparable pairs.
\end{theorem}

\begin{proof}[Proof of Claim \ref{claim:sperner-type}]
Property (a) follows immediately from Sperner's Theorem after noting that $\mc I\cup \mc A$ and $\mc I\cup \mc B$ are independent sets in $G(\mc F)$, and therefore antichains.

Since $\mc I$ consists of isolated sets in $G(\mc F)$, and $\mc S$ was formed by successively removing sets of degree at most $\sqrt{n}$, it follows that the comparability graph $G(\mc I \cup \mc S)$ has at most $\card{\mc S} \sqrt{n}$ edges. By Theorem~\ref{thm:Kleitman}, this forces $\card{\mc I\cup \mc S}\le \binom{n}{\floor{n/2}}+\frac{2}{\sqrt{n}}\card{\mc S}$, and consequently, one has $\card{\mc S} \le 1.01\left(\binom{n}{\floor{n/2}}-\card{\mc I}\right)$ for $n$ sufficiently large, establishing (b).

We finally prove (c).  Let $\mc F' = \mc I \cup \mc S \cup \mc A_{1, \vp} \cup \mc B_{\text{mc}}^{\vp}$.  By the definition of $\mc S$, there are at most $\card{\mc S} \sqrt{n}$ edges incident to $\mc S$ in $G(\mc F')$. Moreover, every vertex from $\mc A_{1,\vp}$ is incident to at most $\sqrt{n}$ vertices from $\mc B_{\text{mc}}^{\vp}$.  As $\mc A$ and $\mc B$ are independent sets, and $\mc I$ is the set of isolated vertices, there are no other edges in $G(\mc F')$. In total, $G(\mc F')$ has at most $\left(\card{\mc S}+\card{\mc A_{1,\vp}}\right) \sqrt{n}$ edges. We thus get
$\card{\mc F'} \le  \binom{n}{\floor{n/2}}+\frac{2}{\sqrt{n}}\left(\card{\mc S}+\card{\mc A_{1,\vp}}\right)$,
by applying Theorem~\ref{thm:Kleitman}. This implies 
\begin{align*}
\binom{n}{\floor{n/2}}-|\mc I| &\ge \left(1-\frac{2}{\sqrt{n}}\right)\left(\card{\mc S}+\card{\mc A_{1,\vp}}\right)+\card{\mc B_{\text{mc}}^{\vp}}\\
&\ge 0.999 \left(\card{\mc S}+\card{\mc A_{1,\vp}}+ \card{\mc B_{\text{mc}}^{\vp}}\right),
\end{align*}
giving 
$\card{\mc S}+\card{\mc A_{1, \vp}} + \card{\mc B_{\text{mc}}^{\vp}} \le 1.01 \left( \binom{n}{\floor{n/2}} - \card{\mc I} \right)$. Moreover, we learn from Claims~\ref{claim:random-subset} (iii) and~\ref{claim:sperner-type} (a) that $\card{\mc A_{2,\vp}}\le 0.01\card{\mc A} \le 0.01\left(\binom{n}{\floor{n/2}}-\card{\mc I}\right)$. Combining these inequalities, we obtain
\[
\card{\mc A_{2,\vp}}+\card{\mc A_{1,\vp}}+\card{\mc B_{\text{mc}}^{\vp}} \le 1.02\left(\binom{n}{\floor{n/2}}-\card{\mc I}\right) - \card{\mc S},
\]
completing the proof.
\end{proof}

\section{Two-colourings without monochromatic \texorpdfstring{$k$}{k}-chains}\label{sec:2-colours}

In this section we prove Theorem~\ref{thm:2-colours}, which we first restate below.

\twocolours*

Recall that this shows the lower bound in~\eqref{ineq:kchainLB} is tight.  The largest $k$-chain-free families, by a result of Erd\H{o}s~\cite{Erd45}, consist of the $k-1$ largest levels of the Boolean lattice, and therefore have size $m_{k-1} = \sum_{i=\floor{(n-k+2)/2}}^{\floor{(n+k-2)/2}} \binom{n}{i}$.  Any larger family would have to contain $k$-chains, and we need to show that these chains place too many restrictions to allow for a larger number of $(2,k)$-colourings.

Just as we did in the proof of Theorem~\ref{thm:3-colours}, we shall do this by partitioning any candidate family in such a way that enables us to bound the number of $(2,k)$-colourings effectively.  However, in this instance, the partition we use is much more complex, and we describe it in the following proposition.  For convenience, we shall use the parameters $\eps = 1/(500k^2)$ and $\omega = 4k \log(1/\eps) / \eps$.

\begin{prop}\label{prop:kchainstructure}
Let $2 \le k \in \bN$, $\eps=1/(500k^2)$ and $\omega = 4k \log(1/\eps) / \eps$.
If $\mc F \subseteq 2^{[n]}$ is a family with at least one $(2,k)$-colouring, then $\mc F$ admits a partition
\[ \mc F = \mc A \cup \mc U \cup \mc D \cup \mc P \cup \mc R, \]
where each part is further partitioned into $k-1$ subparts (e.g. $\mc A = \bigcup_{i=1}^{k-1} \mc A_i$, and similarly for $\mc U, \mc D, \mc P$ and $\mc R$), that satisfies the following properties.
\begin{itemize}
	\item[(P1)] $\mc F$ has at most $2^{\card{\mc A} + \card{\mc U} + \card{\mc D} + \eps \card{\mc R}} 3^{\frac12 \card{\mc P}}$ $(2,k)$-colourings.
	\item[(P2)] For $1 \le i \le k-1$, $\mc A_i$ is an antichain.
	\item[(P3)] If $1 \le i \le k-1$ and $F$ is a set in one of $\mc U_i, \mc U_{i+1}, \mc D_i, \mc D_{i-1}, \mc P_i$ or $\mc R_i$ (where we take $\mc U_k = \mc D_0 = \emptyset$), then $F$ is comparable to at most $2 \omega$ sets in $\mc A_i$.  Moreover, $\mc U_1 = \mc D_{k-1} = \emptyset$.
	\item[(P4)] For each $\mc B_i \in \{ \mc U_i \cup \mc D_i \cup \mc P_i, \mc U_i \cup \mc D_{i-1}, \mc D_i \cup \mc U_{i+1} \}$ (where, again, $\mc U_k = \mc D_0 = \emptyset$), the family $\mc A_i \cup \mc B_i$ has at most $3 \omega \card{\mc B_i}$ comparable pairs.
	\item[(P5)] For any subfamily $\mc H \subseteq \mc F$, there is an antichain $\mc H' \subseteq \mc H$ of size $\card{\mc H'} \ge \card{ \mc H} / (2k-2)$.
\end{itemize}
\end{prop}

As can be seen from the property (P1) above, this partition gives us some control over the number of $(2,k)$-colourings of the set family $\mc F$.  In the next subsection, we shall show how one may combine this with the other properties guaranteed by Proposition~\ref{prop:kchainstructure} to prove Theorem~\ref{thm:2-colours}, thus motivating this complex partition.  Subsection~\ref{subsec:overview} informally explains how the partition will be created, before the proof of Proposition~\ref{prop:kchainstructure} is given in Subsection~\ref{subsec:partition}.  The final subsection is devoted to the proofs of some technical lemmata we shall require.

\subsection{Counting the colourings} \label{subsec:counting}

In this subsection we will show how the partition from Proposition~\ref{prop:kchainstructure} implies Theorem~\ref{thm:2-colours}.  Observe that $\mc A$, which by property (P2) is the union of $k-1$ antichains, is $k$-chain-free, and hence by the theorem of Erd\H{o}s~\cite{Erd45} has size at most $m_{k-1}$.  Indeed, in the extremal configuration, which is the union of the $k-1$ largest levels of the Boolean lattice, each $\mc A_i$ is one of the uniform levels, while we have $\mc U = \mc D = \mc P = \mc R = \emptyset$.

By property (P1), if $\mc F$ is a family with more $(2,k)$-colourings than the union of the $k-1$ largest levels of the Boolean lattice, then at least one of $\mc U, \mc D, \mc P$ or $\mc R$ must be non-empty.  We shall then use properties (P3) and (P4), which bound the number of comparable pairs in certain subfamilies, to obtain upper bounds on the sizes of the parts of the partition, which will in turn show that the number of $(2,k)$-colourings of $\mc F$ is strictly less than $2^{m_{k-1}}$.

This requires the use of supersaturation results: we shall have to deduce from the small numbers of comparable pairs that the corresponding families are small.  Note that Theorem~\ref{thm:Kleitman} is such a result, showing that a family with few comparable pairs cannot be much larger than the middle level of the Boolean lattice.  While that result is tight, our partition separates into $k-1$ levels, and they cannot all occupy the full middle level.  Hence we will have to derive a weighted version of the supersaturation result that is still tight for multiple levels (close to the middle level).

We thus define the weight of a set $F \subseteq 2^{[n]}$ to be
\[ w_k(F) = \min \left\{ \binom{n}{\card{F}}^{-1}, \binom{n}{\floor{\frac{n-k}{2}}}^{-1} \right\}, \]
and the weight of a set family as the sum of the weights of its members, $w_k(\mc F) = \sum_{F \in \mc F} w_k(F)$.  Note that the weight of a set increases with the distance of the set to the middle level, but this increase is capped to prevent undue influence being given to sets that are much smaller or larger than what we expect to find in the optimal construction.

With this notation in place, we present our supersaturation lemma, which we shall prove in Subsection~\ref{subsec:supersat}.

\begin{lemma} \label{lem:2-supersat}
There is some constant $C > 1$ such that the following statement holds for all $\de\in \left( 0,\frac12 \right)$ and all integers $n$ and $k$ with $k\ge 2$ and $n\ge C \de^{-3} k^2$. If $\mc F$ is a subfamily of $2^{[n]}$ with $\weight(\mc F) \ge 1 + r\binom{n}{\floor{n/2}}^{-1}$ for some $r\in \bR$, then the number of comparable pairs in $\mc F$ is at least $\left(\tfrac{1}{2}-\de\right)rn$.
\end{lemma}

Using Lemma~\ref{lem:2-supersat}, it follows that the families in property (P4) of Proposition~\ref{prop:kchainstructure} have small weight.  However, in order to apply property (P1) to bound the number of $(2,k)$-colourings, we will have to control the sizes of these families instead.  The following lemma, whose proof is also in Subsection~\ref{subsec:supersat}, allows us to convert between weights and sizes.

\begin{lemma}\label{lem:transference}
Let $n$ and $k$ be integers with $k\ge 2$ and $n\ge 4k^2$. Suppose $\mc F_0,\mc F_1,\ldots,\mc F_s$ are subfamilies of $2^{[n]}$ such that $\card{\mc F_0}+\sum_{i=1}^{s}\al_i \card{\mc F_i} \ge m_{k-1}+t$
for some positive reals $\al_1,\ldots,\al_s$, and non-negative integer $t$. Then
\[
\weight(\mc F_0)+\left(1+\tfrac{2k^2}{n}\right)\sum_{i=1}^{s}\al_i\weight(\mc F_i) \ge k-1+t \binom{n}{\floor{n/2}}^{-1}.
\] 
\end{lemma}

Armed with these lemmata, together with Proposition~\ref{prop:kchainstructure}, we are in position to prove our theorem.

\begin{proof}[Proof of Theorem~\ref{thm:2-colours}]
Suppose $n \ge C k^4 \log k$, for some constant $C$ large enough to satisfy the inequalities that will follow, and that $\mc F \subseteq 2^{[n]}$ maximises the number of $(2,k)$-colourings.  We apply Proposition~\ref{prop:kchainstructure} to obtain the claimed partition of $\mc F$ for which the properties (P1)--(P5) hold.

In light of the lower bound of~\eqref{ineq:kchainLB}, the number of $(2,k)$-colourings of $\mc F$ must be at least $2^{m_{k-1}}$.  By property (P1), the number of $(2,k)$-colourings of $\mc F$ is at most $2^{\card{\mc A} + \card{\mc U} + \card{\mc D} + \eps \card{\mc R}} 3^{\frac12 \card{\mc P}}$, and thus
\begin{equation}\label{ineq:2-colourSizeSum}
m_{k-1} \le \card{\mc A} + \card{\mc U} + \card{\mc D} + \eps \card{\mc R} + \alpha \card{ \mc P},
\end{equation}
where $\alpha = \frac12 \log 3$.  Our first claim shows that the size of $\mc R$ can be controlled by the sizes of $\mc U, \mc D$ and $\mc P$.

\begin{claim} \label{clm:R-upper}
$\card{\mc R} \le 3k \left( \card{\mc U} + \card{\mc D} + \alpha \card{\mc P} \right)$.
\end{claim}

\begin{proof}
We start by applying the property (P5) to each subfamily $\mc R_i$, obtaining an antichain $\mc R_i' \subseteq \mc R_i$.  Define $\mc R' = \cup_i \mc R_i'$, and note that $\card{\mc R'} \ge \card{\mc R} / (2k-2)$.

Further define $\mc G_i = \mc A_i \cup \mc R_i'$, and consider the comparable pairs in $\mc G_i$.  By (P2) and (P5), both $\mc A_i$ and $\mc R_i'$ are antichains, and hence the only comparabilities in $\mc G_i$ come from $\mc A_i \times \mc R_i'$.  By (P3), there are at most $2 \omega \card{\mc R_i'}$ such pairs.  Applying Lemma~\ref{lem:2-supersat} with $\delta = \frac16$ and $r = 6 \omega n^{-1} \card{\mc R_i'}$, we deduce that $w_k(\mc G_i) \le 1 + 6 \omega n^{-1} \card{\mc R_i'} \binom{n}{\floor{n/2}}^{-1}$.  Letting $\mc G = \cup_i \mc G_i$ and summing over $i \in [k-1]$, we have $w_k(\mc G) \le k-1 + 6 \omega n^{-1} \card{\mc R'} \binom{n}{\floor{n/2}}^{-1}$.

Applying Lemma~\ref{lem:transference} with $s = 1$, $\alpha_1 = 1$, $\mc F_0 = \mc G$ and $\mc F_1 = \emptyset$, and then using~\eqref{ineq:2-colourSizeSum}, we find
\begin{equation} \label{ineq:cardGupper}
\card{\mc G} \le m_{k-1} + 6 \omega n^{-1} \card{\mc R'} \le \card{\mc A} + \card{\mc U} + \card{\mc D} + \eps \card{\mc R} + \alpha \card{\mc P} + 6 \omega n^{-1} \card{\mc R'}.
\end{equation}

Note that $\mc G = \mc A \cup \mc R'$, and thus $\card{\mc G} = \card{\mc A} + \card{\mc R'}$.  Substituting this into~\eqref{ineq:cardGupper} and rearranging gives
\[ \left( 1 - 6 \omega n^{-1} \right) \card{\mc R'} \le \card{\mc U} + \card{\mc D} + \eps \card{\mc R} + \alpha \card{\mc P}. \]

Recalling that $\card{\mc R'} \ge \card{\mc R} / (2k-2)$ and rearranging again, we have the desired bound, since
\[ \frac{1}{3k} \card{\mc R} \le \left( \frac{1 - 6 \omega n^{-1}}{2k-2} - \eps \right) \card{\mc R} \le \card{\mc U} + \card{\mc D} + \alpha \card{\mc P}, \]
where the first inequality follows from the facts that $\eps = 1/(500k^2)$ and $n \ge 18 \omega$.
\end{proof}

By combining~\eqref{ineq:2-colourSizeSum} with Claim~\ref{clm:R-upper}, we obtain
\[ m_{k-1} \le \card{\mc A} + \left(1 + 3k\eps \right) \left( \card{\mc U} + \card{\mc D} + \alpha \card{\mc P} \right). \]

We convert this into a lower bound on the weights of these families by using Lemma~\ref{lem:transference} with $s = 3$, $\mc F_0 = \mc A$, $\mc F_1 = \mc U$, $\mc F_2 = \mc D$, $\mc F_3 = \mc P$, $\alpha_1 = \alpha_2 = 1 + 3 k \eps$, $\alpha_3 = (1 + 3k) \alpha$ and $t = 0$.  This gives
\begin{align}\label{ineq:supersatweights}
k-1 &\le w_k(\mc A) + \left( 1 + \frac{2k^2}{n} \right) \left( 1 + 3k \eps \right) \left( w_k(\mc U) + w_k(\mc D) + \alpha w_k(\mc P) \right) \notag \\
	&\le w_k(\mc A) + \left( 1 + \frac{1}{30k} \right) \left( w_k(\mc U) + w_k(\mc D) \right) + \frac{17}{20} w_k(\mc P).
\end{align}

To complete the proof, we shall obtain an upper bound on the weights of these families as well, with the combination of the two only being satisfied when $\mc U = \mc D = \mc P = \emptyset$.  To this end, for each $1 \le i \le k-1$, let $\mc B_i$ be the family from $\left\{ \mc U_i \cup \mc D_i \cup \mc P_i, \mc U_i \cup \mc D_{i-1}, \mc D_i \cup \mc U_{i+1} \right\}$ that has greatest weight. Our next claim shows that these weights cannot be too small.

\begin{claim}\label{clm:supersatgain}
$\sum_i w_k(\mc B_i)\ge \left(1 + \frac{1}{20k} \right) \left( w_k(\mc U) + w_k (\mc D) \right) + \frac{9}{10} w_k( \mc P).$
\end{claim}

\begin{proof}
Since we could always have chosen $\mc B_i = \mc U_i \cup \mc D_i \cup \mc P_i$, we must have $w_k(\mc B_i) \ge w_k( \mc U_i) + w_k(\mc D_i) + w_k ( \mc P_i)$.  Summing these inequalities over all $i$, we have
\begin{equation} \label{ineq:claim2bound1}
\sum_{i=1}^{k-1} w_k(\mc B_i) \ge w_k(\mc U) + w_k(\mc D) + w_k(\mc P).
\end{equation}

On the other hand, consider the families in $\left\{ \mc U_i, \mc D_i : i \in [k-1] \right\}$, and note by (P3) we have $\mc U_1 = \mc D_{k-1} = \emptyset$.  Suppose, for some $j \ge 2$, $\mc U_j$ has the greatest weight out of these families (a similar argument applies when some $\mc D_j$, $j \le k-2$, is the heaviest).  As the heaviest of the families, we must have $w_k(\mc U_j) \ge \frac{1}{2k}(w_k(\mc U) + w_k(\mc D))$.

We can again bound the total weight of the families $\mc B_i$ from below by the following choices.  For $i \le j-1$, take $\mc B_i = \mc D_i \cup \mc U_{i+1}$, and $\mc B_i = \mc U_i \cup \mc D_i \cup \mc P_i$ when $i \ge j$.  We would then have
\begin{align} \label{ineq:claim2bound2}
\sum_{i=1}^{k-1} w_k(\mc B_i) &\ge \sum_{i=1}^{j-1} \left( w_k(\mc D_i) + w_k(\mc U_i) \right) + \sum_{i=j}^{k-1} \left( w_k( \mc U_i) + w_k (\mc D_i) + w_k (\mc P_i) \right) \notag \\
&\ge \sum_{i=2}^{k-1} w_k( \mc U_i) + \sum_{i=1}^{k-2} w_k( \mc D_i) + w_k(\mc U_j) = w_k(\mc U) + w_k(\mc D) + w_k(\mc U_j). \notag \\
&\ge \left( 1 + \frac{1}{2k} \right) \left( w_k(\mc U) + w_k(\mc D) \right).
\end{align}

Taking a convex combination of these lower bounds, with coefficient $\tfrac{9}{10}$ for~\eqref{ineq:claim2bound1} and $\tfrac{1}{10}$ for~\eqref{ineq:claim2bound2}, we arrive at the claimed lower bound on $\sum_i w_k(\mc B_i)$.
\end{proof}

To finish the proof, we now bound the weights of the families $\mc B_i$ from above.  Using (P4), we observe that the number of comparable pairs in the family $\mc A_i \cup \mc B_i$ is at most
\[ 3 \omega \card{\mc B_i} \le 3 \omega \left( \card{\mc U_i} + \card{\mc U_{i+1}} + \card{\mc D_i} + \card{\mc D_{i-1}} + \card{\mc P_i} \right). \]
Applying Lemma~\ref{lem:2-supersat} with $\delta = \tfrac16$ and $r = 9 \omega n^{-1} \left( \card{\mc U_i} + \card{\mc U_{i+1}} + \card{\mc D_i} + \card{\mc D_{i-1}} + \card{\mc P_i} \right)$, we find
\begin{align*}
w_k(\mc A_i \cup \mc B_i) &\le 1 + 9 \omega n^{-1} \left( \card{\mc U_i} + \card{\mc U_{i+1}} + \card{\mc D_i} + \card{\mc D_{i-1}} + \card{\mc P_i} \right) \binom{n}{\floor{n/2}}^{-1} \\
&\le 1 + 9 \omega n^{-1} \left( w_k(\mc U_i ) + w_k (\mc U_{i+1} ) + w_k (\mc D_i) + w_k( \mc D_{i-1} ) + w_k( \mc P_i ) \right),
\end{align*}
where the second inequality follows from the fact that every set in $2^{[n]}$ has weight at least $\binom{n}{\floor{n/2}}^{-1}$.  Summing over $1 \le i \le k-1$, 
\begin{equation} \label{ineq:Bweightupper}
	w_k(\mc A) + \sum_{i=1}^{k-1} w_k(\mc B_i) = \sum_{i=1}^{k-1} w_k (\mc A_i \cup \mc B_i) \le k-1 + 9 \omega n^{-1} \left( 2 w_k(\mc U) + 2 w_k(\mc D) + w_k(\mc P) \right).
\end{equation}

On the other hand, we have the lower bound
\begin{align} \label{ineq:Bweightlower}
	w_k(\mc A) + \sum_{i=1}^{k-1} w_k(\mc B_i) &\ge w_k(\mc A) + \left(1 + \tfrac{1}{20k} \right) \left(w_k( \mc U) + w_k( \mc D) \right) + \tfrac{9}{10} w_k(\mc P) \notag \\
	&\ge k-1 + \tfrac{1}{60k} \left( w_k(\mc U) + w_k(\mc D) \right) + \tfrac{1}{20} w_k(\mc P),
\end{align}
where we use Claim~\ref{clm:supersatgain} for the first inequality and~\eqref{ineq:supersatweights} for the second.  As $n > 1080 \omega k$,~\eqref{ineq:Bweightupper} and~\eqref{ineq:Bweightlower} can only be simultaneously satisfied when $w_k(\mc U) = w_k(\mc D) = w_k(\mc P) = 0$.  This in turn implies $\mc U = \mc D = \mc P = \emptyset$.  By Claim~\ref{clm:R-upper}, it follows that $\mc R = \emptyset$ as well.  Hence $\mc F = \mc A$, the union of $k-1$ antichains.  Thus $\mc F$ is a $k$-chain-free family, which can have size at most $m_{k-1}$, with at most $2^{m_{k-1}}$ two-colourings.  This completes the proof of Theorem~\ref{thm:2-colours}.
\end{proof}

\subsection{Forming the partition} \label{subsec:partition}

Now that we have seen how the partition from Proposition~\ref{prop:kchainstructure} can be used to establish Theorem~\ref{thm:2-colours}, we shall describe how a partition with properties (P1)--(P5) can be formed.  We begin with an informal overview of the process, before providing a detailed proof of the proposition.

\subsubsection{An overview} \label{subsec:overview}

For inspiration, we first consider the optimal construction, which is the union of the $k-1$ largest uniform levels of the Boolean lattice.  This family admits a natural partition into $k-1$ antichains.  Moreover, each set contains many sets from the antichain below, and is contained in many sets from the antichain above.

We shall endeavour to build a similar structure --- a sequence of antichains, with each set containing many sets from the antichain below, and being contained in many sets from the antichain above.  This large degree in the comparability graph $G(\mc F)$ (as defined in Section~\ref{subsec:orgnot}) will be important for us, as we will be able to use it to create many chains.  If we find a $(k-1)$-chain $F_1 \subset F_2 \subset \hdots \subset F_{k-1}$ that is monochromatic in many colourings, then it follows that any sets that contain $F_{k-1}$ must all have the opposite colour.  This allows us to remove all such sets into a separate family $\mc R$.  As the colour of these sets is determined, we do not lose many colourings of the entire family $\mc F$ when we do so.  Similarly, any sets that are contained in $F_1$ could also be removed.

Sometimes, though, we may encounter sets that do not have the desired large degrees into their neighbouring antichains.  We take such sets out of their antichains, and place them in separate families.  If they are contained in too few sets from the antichain above, we call them \emph{up-sparse}, and place them in the family $\mc U$.  If they contain too few sets from the antichain below, we name them \emph{down-sparse}, and place them in $\mc D$.  As we have seen in the previous subsection, their low degrees into the (eventual) antichains will allow us to exploit the supersaturation result of Lemma~\ref{lem:2-supersat}.

At this stage in the process, the parts $\mc A_1$ up to $\mc A_{k-1}$ may not actually be antichains, but could contain a few comparable pairs.  Our next step is to remove these pairs to form genuine antichains.  If a pair is often monochromatic, then we can extend it into a monochromatic chain, which again allows us to remove a large number of sets into $\mc R$ without losing many colourings.  On the other hand, if a pair is more often oppositely-coloured, then we place it in the family $\mc P$ instead.  Note that for the pairs in $\mc P$, we know that out of the four ways the two sets could have been coloured, two are much more likely to occur, which allows us to better bound the number of colourings with respect to the sets in $\mc P$.

This explains the key ideas behind the formation of the partition, as well as the roles played by the parts $\mc R$, $\mc U$, $\mc D$ and $\mc P$.  Once we have completed the steps described above, there will be some final cleaning of the partition to ensure that all the properties (P1)--(P5) hold, after which the proof of Proposition~\ref{prop:kchainstructure} will be complete.  We now proceed to the details of the procedure.

\subsubsection{The detailed procedure}

To begin, observe that if $\mc F$ contained a $(2k-1)$-chain $\mc C$, then in any two-colouring of $\mc F$, $\mc C$ would contain a monochromatic $k$-chain.  Hence, if $\mc F$ admits even a single $(2,k)$-colouring, we deduce that $\mc F$ must be $(2k-1)$-chain-free. Appealing to Mirsky's theorem \cite{Mir71}, we see that $\mc F$ can be partitioned into $2k-2$ antichains.  Explicitly, let $\mc A_i$ be the maximal elements in the poset $\mc F \setminus \left( \cup_{j=1}^{i-1} \mc A_j \right)$, which gives a partition 
\begin{equation} \label{eqn:initialPart}
\mc F = \mc A_1 \cup \hdots \cup \mc A_{2k-2}
\end{equation}
into $2k-2$ antichains (some of which may be empty).  This partition has the following quality.

\begin{itemize}
	\item[(Q1)] If $i < j$, $F \in \mc A_i$ and $G \in \mc A_j$, then $F \not\subset G$.
\end{itemize}

(Q1) obviously holds for the partition described above.  However, our partition shall be dynamic, as we will move sets between parts to reach the desired final partition.  We shall ensure (Q1) is maintained throughout the process.  It will also be convenient to fix a linear extension $(\mc F, \prec)$, such that given sets $F, G \in \mc F$, we have $F \subset G$ only if $F \prec G$.  We further require that the linear extension starts with all sets in $\mc A_{2k-2}$, followed by those in $\mc A_{2k-3}$, and so on, until the sets in $\mc A_1$ are listed last.

As we proceed, we will denote by $\fr C$ the set of $(2,k)$-colourings of $\mc F$ currently under consideration. At the start of the process, $\fr C$ will contain all $(2,k)$-colourings of $\mc F$.  However, we shall occasionally colour some sets in $\mc F$, and will only retain in $\fr C$ those colourings that agree with the partial colouring of $\mc F$.  One of the ways we shall colour sets is through the branching operation, which we will now describe.  The main idea is to build chains until we find a large number of sets that must often be monochromatic, thus allowing us to colour and remove many sets while only shrinking $\fr C$ moderately.\footnote{For an alternative entropic viewpoint, we could imagine that an adversary has chosen a $(2,k)$-colouring of $\mc F$, which we seek to determine by asking a series of questions.  Our goal will then be to limit the amount of information we receive, as we can then bound the number of $(2,k)$-colourings from which the adversary could choose her colouring.} \\

\noindent \emph{Branching from a coloured vertex}:  Recall that $\eps = 1/(500k^2)$ and $\omega = 4k \log(1/\eps) / \eps$.  Assume some set $A \in \mc F$ has been coloured, say red, and there is some index $\ell$ such that $d^+(A, \mc A_{\ell}) > \omega$ (or $d^-(A, \mc A_{\ell}) > \omega$).  Because $\sum_{j \ge 0} \left(1 - 2^{-\eps} \right) 2^{- \eps j} = 1$, one of the two following statements must hold:
\begin{itemize}
	\item[(i)] In at least a $(1 - 2^{-\eps})$-fraction of the colourings in $\fr C$, all sets in $N^+(A, \mc A_{\ell})$ (respectively, $N^-(A, \mc A_{\ell})$) are all coloured blue.  Note that this corresponds to the $j=0$ summand above.
	\item[(ii)] There exists an integer $j \ge 1$ such that in at least a $(1 - 2^{-\eps})2^{- \eps j}$-fraction of all colourings in $\fr C$, the $j$th set with respect to $\prec$ in $N^+(A, \mc A_{\ell})$ (respectively, $N^-(A, \mc A_{\ell})$), say $A_{\ell}$, is the first red set in $N^+(A, \mc A_{\ell})$ (respectively, $N^-(A, \mc A_{\ell})$).
\end{itemize}
In the first case, we can colour all the sets in $N^+(A, \mc A_{\ell})$ (respectively, $N^-(A, \mc A_{\ell})$) blue, remove them from $\mc A_{\ell}$ and place them in $\mc R$, and restrict $\fr C$ to the $(1 - 2^{-\eps})\card{\fr C}$ colourings where these sets are all blue.  In the latter case, we colour $B$ red and the preceding sets in $N^+(A, \mc A_{\ell})$ (respectively, $N^-(A, \mc A_{\ell})$) blue, remove these sets to $\mc R$, and restrict $\fr C$ to the remaining $(1 - 2^{-\eps}) 2^{- \eps j}|\fr C|$ colourings. \\

We can piece together these branching steps into the following operations. \\

\noindent \emph{Branching up}:  Let $A$ be a coloured set in $\mc F$, and let $\ell \ge k-1$ be an integer such that $d^+(A, \mc A_{\ell}) > \omega$ and $d^+(F, \mc A_{i-1}) > \omega$ for all $2 \le i \le \ell$ and $F \in \mc A_i$.  We branch from $A$ to its supersets in $\mc A_{\ell}$.  When case (ii) occurs, we find the superset $A_{\ell} \in \mc A_{\ell}$ which often has the same colour as $A$, we iterate, branching from $A_{\ell}$ to its supersets in the next level. Repeating this to get a chain $A \subset A_{\ell} \subset A_{\ell +1} \subset \ldots$ until case (i) occurs, with all sets in the next level having the opposite colour. \\

\noindent \emph{Branching down}:  Let $A$ be a coloured set in $\mc F$ with $d^-(A, \mc A_1) > \omega$, and suppose $d^-(F, \mc A_{i+1}) > \omega$ for all $1 \le i \le k-2$ and $F \in \mc A_i$.  Branching down works exactly like branching up, except we consider subsets in the lower level, instead of supersets from the level above. \\

\noindent \emph{Branching up and down}:  Suppose some $2 \le i_0 \le k-2$ has the property that $d^+(F, \mc A_{i-1}) > \omega$ for every $2 \le i \le i_0$ and $F \in \mc A_i$, and $d^-(F, \mc A_{i+1}) > \omega$ for every $i_0 + 1 \le i \le k-2$ and $F \in \mc A_i$.  Let $A$ be a coloured set with $d^+(A, \mc A_{i_0}) > \omega$ and $d^-(A, \mc A_{i_0+1}) > \omega$.  We first branch up from $A$, passing through its supersets in $\mc A_{i_0}$, and continuing through the higher parts.  If we have not encountered a monochromatic neighbourhood by the time we reach $\mc A_1$, we branch down from $A$, considering its subsets in $\mc A_{i_0 + 1}$, and then continuing through the lower parts.  We can also branch up and down from a coloured pair $A \subset B$ in $\mc A_{i_0 + 1}$ if $d^+(B, \mc A_{i_0}) > \omega$ and, as supposed, $d^-(A, \mc A_{i_0+2}) > \omega$.  We then branch up through supersets of $B$ in higher parts, and then branch down through subsets of $A$ in lower parts. \\

Note that with each branching step in these operations, whenever we encounter case (ii), we find a longer monochromatic chain involving the initial set $A$.  Thus these operations must terminate within $k-1$ branching steps.  The following lemma quantifies the outcome of a branching operation.

\begin{lemma}\label{lem:branch-colour}
If, when branching up, down or up and down from a coloured set $A$ or a pair $A \subset B$, $t$ further sets are added to $\mc R$, then $\fr C$ (at the end of the operation) shrinks  by a factor of at most $\tfrac16 2^{\eps t}$ (comparing to $\fr C$ at the beginning of the operation).
\end{lemma}

\begin{proof}
Assume we encounter case (ii) $s$ times before encountering case (i) and ending the branching operation.  As we argued previously, there can be at most $k-1$ branching steps, and therefore we must have $s \le k-2$.  Let $j_1, j_2, \hdots, j_s$ be the indices of the monochromatic neighbours given in each instance of case (ii).  The number of sets moved to $\mc R$ is then $t \ge j_1 + j_2 + \hdots + j_s + \omega$. Moreover, the set of $(2,k)$-colourings $\fr C$ shrinks by a factor of at most
\begin{align*}
\left((1-2^{-\eps})\prod_{i=1}^s (1-2^{-\eps})2^{-\eps j_i}\right)^{-1} &= 2^{\eps(j_1+\hdots + j_s)}(1-2^{-\eps})^{-s-1} \le 2^{\eps (t - \omega)} \left( \frac{ 2^{\eps}}{2^{\eps} - 1} \right)^{s+1} \\
&\le 2^{\eps( t - \omega)} \left( 2^{\eps} \eps^{-2} \right)^{s+1} < 2^{\eps(t - \omega + k)} \eps^{-2k} = 2^{\eps t} ( \eps^2 2^{\eps} )^k \le \tfrac16 2^{\eps t},
\end{align*}
where the first inequality holds since $2^x \ge 1 + x^2$ for every $x \in [0, 1]$, the second inequality since $s + 1 < k$, and the following equality is due to the fact that $\omega = 4k \log(1 / \eps) / \eps$.
\end{proof}

Having defined these branching operations, we are now in a position to specify how we obtain the partition of Proposition~\ref{prop:kchainstructure}.  The procedure consists for four stages.

\paragraph{Stage I.}

The goal of this stage is to compress the $2k-2$ antichains in~\eqref{eqn:initialPart} into $k-1$ parts $\mc A_1, \hdots, \mc A_{k-1}$.  These will no longer necessarily be antichains, but we shall ensure no set is contained in more than $\omega$ other sets from its own part.  We do this greedily, shifting a set up to a higher part whenever possible, and then using branching operations when needed. \\

\noindent \emph{Shifting up}: Running $i$ from $2$ to $2k-2$, consider the sets in $\mc A_i$ in the reverse of the predetermined linear order $\prec$.  When the set $F$ is being considered, if $d^{+}(F,\mc A_{i-1}) \le \omega$, move $F$ up to $\mc A_{i-1}$. \\

We repeatedly run the shifting up operation until no set is moved.  At this point, every set is contained in more than $\om$ sets from the part directly above. As a consequence, if $\mc A_i$ is non-empty for some $i \ge k$, the sets in $\mc A_i$ are in many $k$-chains.  We shall use this fact to colour and remove many sets efficiently via the following operation. \\

\noindent \emph{Colouring and branching up}: Let $\ell =\max\{i:\mc A_i\ne \emptyset\}$. If $\ell \ge k$, let $F_1 \in \mc A_{\ell}$ be the first set in $\mc A_{\ell}$ with respect to $\prec$. Colour $F_1$ with whichever colour occurs most frequently in $\fr C$ (breaking ties arbitrarily), say red.  Restrict $\fr C$ to those colourings where $F_1$ is red and move $F_1$ to $\mc R$. We then branch up from $F_1$ to its supersets in $\mc A_{\ell - 1}$, and onwards through higher parts. \\

After running through this operation, the sets removed may leave space from some lower sets to be shifted up.  Hence we repeat this sequence of operations until $\cup_{i \ge k} \mc A_i$ is empty, which marks the end of Stage I. At this point, we are left with the partition $\mc F = \mc A_1 \cup \hdots \cup \mc A_{k-1} \cup \mc R$, with the parts $\mc A_i$ having the following two qualities in addition to (Q1) from before.

\begin{itemize}
	\item[(Q2)] If $F \in \mc A_i$, then $d^{+}(F,\mc A_i) \le \omega$.
	\item[(Q3)] If $F \in \mc A_i$ for some $2 \le i \le k-1$, then  $d^{+}(F,\mc A_{i-1}) > \omega$.
\end{itemize}

The following lemma shows these attributes do indeed hold after Stage I, and that we do not restrict our set of colourings $\fr C$ too greatly.

\begin{lemma} \label{lem:end-I}
At the end of Stage I, we have the partition $\mc F=\mc A_1\cup \hdots \cup \mc A_{k-1}\cup \mc R$ with the qualities (Q1)--(Q3).  Moreover, the set $\fr C$ of colourings shrinks by a factor of at most $2^{\eps \card{\mc R}}$.
\end{lemma}

\begin{proof}
We first prove by contradiction that our algorithm preserves the monotonicity of (Q1). Suppose we are at the step when monotonicity is violated, and $(A,B)\in \mc A_i\times\mc A_j$ is the pair with $i<j$ and $A\subset B$. It must be that $j=i+1$ and $A$ has been moved from $\mc A_j$ to $\mc A_i$ in the previous step via a shifting up operation, which implies $d^{+}(A,\mc A_i) \le \omega$.  However, since $A \subset B$, we must have $d^+(B, \mc A_i) \le \omega$ as well, and $A \prec B$.  This means that $B$ would have been moved from $\mc A_j$ to $\mc A_i$ before $A$, a contradiction.

We proceed by showing that (Q2) is maintained at every step of Stage I. If this is not true, let us look closer at the first time when there exists $F\in \mc A_i$ with $d^{+}(F,\mc A_i)>\om$. Initially, $\mc A_j$ is an antichain for every $j \ge 1$, so (Q2) holds at the beginning of Stage I. Hence either $F$ or a set $F'$ in $N^{+}(F,\mc A_i)$ has been moved up from $\mc A_{i+1}$ to $\mc A_i$ in the previous step. The former case cannot happen because we only shift up if $d^{+}(F,\mc A_i)\le\om$. On the other hand, by (Q1), we could not have had $F\in \mc A_i$ and $F'\in \mc A_{i+1}$ for any $F'\supset F$, and so the latter case is ruled out as well.

As Stage I stops only when we can no longer apply the shifting up operation, our partition satisfies (Q3) at the end of Stage I.

Finally, consider the colouring and branching up operations that we perform.  When we colour the first set $F_1$, we choose the colour it most frequently receives in the colourings of $\fr C$, which causes $\fr C$ to shrink by a factor of at most two.  By Lemma~\ref{lem:branch-colour}, if $t$ more sets are added to $\mc R$ by this operation, $\fr C$ shrinks by a further factor of at most $\tfrac16 2^{\eps t}$.  Hence when $\mc R$ grows by a total of $t'$ sets, $\fr C$ shrinks by a factor of at most $2^{\eps t'}$, thus ensuring that $\fr C$ is at most $2^{\eps \card{\mc R}}$ times smaller by the end of Stage I.
\end{proof}

\paragraph{Stage II.}
The previous stage gave us good control, through (Q1)--(Q3), of the number of sets a given set is contained in. The goal of this stage is to obtain similar control over the number of sets a given set \emph{contains}. More precisely, at the end of Stage II, we shall have a partition $\mc F=\bigcup_{i\in [k-1]}\left(\mc A_i\cup\mc U_i \cup \mc D_i \right)\cup \mc R$ that has the three following characteristics in addition to (Q1)--(Q3).
\begin{itemize}
	\item[(Q4)] If $F \in \mc A_i$, then $d^{-}(F,\mc A_i)\le \omega$.
	\item[(Q5)] If $F \in \mc A_i$ for some $1 \le i \le k-2$, then $d^{-}(F,\mc A_{i+1})>\omega$.
	\item[(Q6)] If $U \in \mc U_i$ for some $2 \le i \le k-1$, then $\max \left\{ d(U, \mc A_i), d(U, \mc A_{i-1}) \right\} \le 2 \omega$.  If $D \in \mc D_i$ for some $1 \le i \le k-2$, then $\max \left\{ d(D, \mc A_i), d(D, \mc A_{i+1}) \right\} \le 2 \omega$.  Moreover, $\mc U_1 = \mc D_{k-1} = \emptyset$.
\end{itemize}

Stage II consists of two substages, of which we now specify the first.

\paragraph{Stage IIa.}
Set $\mc U_i=\mc D_i=\emptyset$ for every $i\in [k-1]$. We process the sets in $\mc A$ one-by-one, starting with the sets in $\mc A_{k-1}$ and working up to the sets in $\mc A_1$.  Within a part $\mc A_i$, we consider the sets according to the order $\prec$.  Assume we are currently considering a set $F \in \mc A_i$.  There are three possibilities.

{
\narrower
\paragraph{Case 1:} $d^-(F, \mc A_i) \le \omega$ and, if $i \le k-2$, $d^-(F, \mc A_{i+1}) > \omega$.

Since $F$ has qualities (Q4) and (Q5), we leave $F$ in place and proceed to the next set.

\paragraph{Case 2:} $i \le k-2$, $d^{-}(F, \mc A_i) \le \omega$ and $d^{-}(F, \mc A_{i+1}) \le \omega$.

In this case, $F$ has quality (Q4), not containing many sets from its own part.  However, it fails (Q5), as it also doesn't contain many sets from the level below.  Hence we remove it from $\mc A_i$, and place it in the part $\mc D_i$ instead, to indicate it is \emph{down-sparse}.  We then proceed to the next set.

\paragraph{Case 3:} $d^-(F, \mc A_i) > \omega$.

In the final case, $F$ contains many sets from its own part.  It is therefore contained in several $k$-chains (see Remark~\ref{rmk:case3} below), which we can use to efficiently colour and remove sets.  However, this can cause some of the previously established qualities to be violated, and thus we must restart the stage after the branching operation.  We explain in more detail below. \\

\noindent \emph{Colouring and branching up and down}: Colour $F$ with its most frequent colour in $\fr C$, say red.  Restrict $\fr C$ to those colourings where $F$ is red and remove $F$ to $\mc R$. We then apply the branching up and down operation from $F$ with $i_0 = i-1$. 

\begin{rmk}\label{rmk:case3}
Note that for any set $F' \supset F$ we must have $F \prec F'$, and therefore $F'$ is yet to be considered, and is thus still in $\mc A$. This, together with (Q3), implies we can branch up.  On the other hand, for a set $F' \subset F$, we have $F' \prec F$, and would therefore already have considered $F'$ in Stage IIa. If $F'$ is still in $\mc A_j$ for some $i \le j \le k-2$ (in particular, if $F$ falls under Case 3), $F'$ must have fallen under Case 1, and thus contains more than $\omega$ sets in $\mc A_{j+1}$, thereby ensuring that the branching down part of the operation can also be carried out.\footnote{In particular, note that if $\mc A_{k-1}$ becomes empty at some point, then the first set we consider, which is minimal in its part $\mc A_i$, would fall under Case 2, and thus be removed to $\mc D_i$. This would then continue to occur for each subsequent set, and so at the end of Stage IIa, all parts $\mc A_i$ would be empty, with their members having been placed in the parts $\mc D_i$ instead.}
\end{rmk}

However, when we remove sets in the branching operation, we could destroy the qualities (Q3) and (Q5) of sets that we have already considered, as the up- and down-degrees could decrease, and hence we must restore these.  For $1 \le i \le k-1$, we return any sets in $\mc D_i$ to the part $\mc A_i$.  Note that sets which were coloured and removed to $\mc R$ are \emph{not} returned.  With $\mc D$ once again empty, we shift sets up to restore the quality (Q3), just as we did in Stage I. \\

\noindent \emph{Shifting up}: Starting from $i = 2$ and running through to $i = k-1$, consider the sets in $\mc A_i$ in the reverse of the order $\prec$.  When dealing with $F \in \mc A_i$, if $d^+(F, \mc A_{i-1}) \le \omega$, we move $F$ to the part $\mc A_{i-1}$.  Repeat this procedure until no sets are moved. \\

After the shifting up procedure terminates, we have again restored (Q1)--(Q3).  At this point, we can restart Stage II, aiming to ensure (Q4) and (Q5) are satisfied as well.  Hence, to end Case 3, we return to the first set under $\prec$ in $\mc A_{k-1}$ and restart Stage IIa. \\

}

Stage IIa ends once we have gone through all the sets in $\mc A$ without encountering Case 3 (which would cause us to restart the stage), after which we proceed to the second part of Stage II.  Observe that $\mc R$ grows monotonically with each iteration, and hence we can only restart Stage IIa a finite number of times before moving on to Stage IIb.

\paragraph{Stage IIb.} The partition $\mc F=\bigcup_{i\in [k-1]}\left(\mc A_i\cup\mc D_i\right)\cup \mc R$ at the end of Stage IIa may no longer have (Q3), as supersets of sets could have been removed from $\mc A$. To fix this issue, we move any sets in $\mc A$ that violate either (Q3) or (Q5) to $\mc U$ or $\mc D$ respectively. \\

\noindent \emph{Moving to $\mc U \cup \mc D$}:  We process the sets in $\mc A$ according to the order $\prec$.  Given $F \in \mc A_i$, if $i \ge 2$ and $d^+(F, \mc A_{i-1}) \le \omega$, move $F$ to $\mc U_i$.  If $i \le k-2$ and $d^-(F, \mc A_{i+1}) \le \omega$, move $F$ to $\mc D_i$. Repeat this process until no further sets are moved. \\

We repeatedly apply this operation until no sets are moved, which marks the end of Stage IIb and, with it, Stage II.  Lemma~\ref{lem:end-II} sumarises the effects of Stage II on our partition, but we first introduce some terminology that we will use in the remainder of this section.

\begin{dfn}\label{def:prosp-nhd}
Given $i\in [k-1]$ and a subfamily $\mc B_i$ of $\mc U_i \cup \mc D_i \cup \mc P_i$, the \emph{right-to-left order} of the members of $\mc A_i \cup \mc B_i$ is obtained by placing from right to left the sets of $\mc B_i$ in the order in which they entered $\mc B_i$, followed by elements of $\mc A_i$ according to the order $\prec$.  Moreover, consider the moment some set $F$ is (last\footnote{Note that $\mc P_i = \emptyset$ throughout Stage II.  Furthermore, the only way a set can leave $\mc U \cup \mc D \cup \mc P$ is through Case 3 of Stage IIa.  That apart, if a set is moved into one of these parts, it will remain there.}) moved to $\mc U_i\cup\mc D_i\cup\mc P_i$. The \emph{prospective left-neighbourhood} of $F$, denoted by $N_{\textup{pr}}(F)$, is defined to be $N(F,\mc A_i\cup\mc A_{i-1})$ if $F\in \mc U_i$, $N(F,\mc A_i\cup\mc A_{i+1})$ if $F\in \mc D_i$, and $N(F,\mc A_i)$ if $F\in\mc P_i$, where the parts $\mc A_{i-1}, \mc A_i$ and $\mc A_{i+1}$ are taken as they were during the step when $F$ was removed from $\mc A$.  Even though these parts could shrink during subsequent steps of the process, we will \emph{not} update $N_{\textup{pr}}(F)$ correspondingly. 
\end{dfn}

\begin{lemma}\label{lem:end-II}
At the end of Stage II, the following statements are true.
\begin{itemize}
\item[\rm (i)] The partition $\mc F=\bigcup_{i\in [k-1]}\left(\mc A_i\cup\mc U_i\cup\mc D_i\right)\cup \mc R$ has the qualities (Q1)--(Q6).
\item[\rm (ii)] For every $F \in \mc U \cup \mc D$, $\card{N_{\textup{pr}}(F)} \le 3 \omega$.  Moreover, suppose $\mc B_i \in \left\{ \mc U_i \cup \mc D_i, \mc U_i \cup \mc D_{i-1}, \mc D_i \cup \mc U_{i+1} \right\}$.  If $F \in \mc B_i$, then the left-neighbourhood of $F$ with respect to the right-to-left order of the members of $\mc A_i \cup \mc B_i$ is a subfamily of $N_{\textup{pr}}(F)$.
\item[\rm (iii)] If $t$ sets are coloured and removed during this stage, the family $\fr C$ of colourings shrinks by a factor of at most $2^{\eps t}$.
\end{itemize}
 
\end{lemma}
\begin{proof}
For (i), we verify the various qualities in turn, starting with (Q1) and (Q2).  At the beginning of Stage II, the partition satisfies these qualities, which can only be violated if sets are moved into parts $\mc A_i$. This only occurs during Case 3 of Stage IIa, when shifting up and when emptying $\mc D$ before restarting the stage.  As proved in Lemma~\ref{lem:end-I}, the shifting up procedure preserves (Q1) and (Q2).  Furthermore, in Stage IIa, when $\mc D \neq \emptyset$, no sets are moved within or into $\mc A$. Hence when the sets from $\mc D$ are returned to their original parts in $\mc A$, they still have (Q1) and (Q2).  This shows that these qualities hold at the end of Stage II.

As Stage IIb only terminates when (Q3) and (Q5) are satisfied for every set in $\mc A$, it is evident that these qualities hold.  (Q4) holds for every set in $\mc A$ at the end of Stage IIa, since otherwise we would have fallen in Case 3 and restarted the stage.  As sets can only be removed from $\mc A$ in Stage IIb, it follows that the degree of a set into its own part cannot increase, and therefore (Q4) still holds at the end of Stage II.

To establish (Q6), first note that we never put any sets into $\mc U_1$ or $\mc D_{k-1}$, and hence these parts must be empty.  Now suppose $i \le k-2$ and $D \in \mc D_i$, and consider the step of Stage II in which $D$ was last moved from $\mc A_i$ to $\mc D_i$.  At that time, by (Q2), we had $d^+(D, \mc A_i) \le \omega$.  Since $D$ was not coloured and removed in Stage IIa, we must also have had $d^-(D, \mc A_i) \le \omega$, and thus $d(D, \mc A_i) = d^+(D, \mc A_i) + d^-(D, \mc A_i) \le 2 \omega$.  Sets can only be added to $\mc A_i$ in Case 3 of Stage IIa, which cannot have occurred after $D$ last entered $\mc D_i$, since $\mc D_i$ would have been emptied in this case.  Hence $d(D, \mc A_i) \le 2 \omega$ at the end of Stage II as well.  Also, we would only place $D$ in $\mc D_i$ if $d^-(D, \mc A_{i+1}) \le \omega$, and by (Q1) we have $d^+(D, \mc A_{i+1}) = 0$.  Hence $d(D, \mc A_{i+1}) \le \omega$, and again this degree would not have increased in later steps.  This shows that (Q6) holds for $D$, and a similar argument applies when $i \ge 2$ and $U \in \mc U_i$.  This completes the proof of (i).

For the first assertion in (ii), note that for $F \in \mc D_i$, $N_{\textup{pr}}(F)$ is the neighbourhood of $F$ in $\mc A_i \cup \mc A_{i+1}$ at the time when $F$ was placed in $\mc D_i$.  As argued when establishing (Q6), this neighbourhood is $N^-(F, \mc A_i) \cup N^+(F, \mc A_i) \cup N^-(F, \mc A_{i+1})$, and each of these sets has size at most $\omega$.  Thus $\card{N_{\textup{pr}}(F)} \le 3 \omega$.  A very similar argument applies when $F \in \mc U_i$ instead.

For the second assertion in (ii), we consider the case $\mc B_i = \mc U_i \cup \mc D_{i-1}$, as the other two cases can be handled similarly.  Let $F \in \mc B_i$, and observe that in the right-to-left order of $\mc A_i \cup \mc B_i$, all sets to the left of $F$ were in $\mc A_i \cup \mc A_{i-1}$ when $F$ was added to $\mc B_i$.  Since, for $F \in \mc B_i$, $N_{\textup{pr}}(F)$ is defined to be precisely the neighbours of $F$ in $\mc A_i \cup \mc A_{i-1}$ at the time $F$ was removed from $\mc A$, it follows that the left-neighbours of $F$ in $\mc A_i \cup \mc B_i$ are all contained in $N_{\textup{pr}}(F)$.

Finally, the claim in (iii) follows directly from Lemma~\ref{lem:branch-colour}.
\end{proof}

\paragraph{Stage III.}  Stage II gave us very good control over the degrees in the parts $\mc A_i$, ensuring that sets were comparable to very few other sets from their own parts, but were contained in many sets from the part above and contained many sets from the part below.  In this stage we take another step towards the desired final partition by ensuring the parts $\mc A_i$, rather than merely being sparse in comparable pairs, will be antichains.  This shall provide us with the following additional quality.
\begin{itemize}
	\item[(Q7)] The parts $\mc A_i$, $1 \le i \le k-1$, are antichains.
\end{itemize}

We achieve this by removing comparable pairs within the parts $\mc A_i$ in one of two ways.  We shall branch up and down from monochromatic pairs, efficiently colouring and moving sets to $\mc R$.  On the other hand, if the pair is oppositely-coloured most of the time, then we shall place it in the hitherto empty part $\mc P_i$.  Sets moved to these parts have the following quality.
\begin{itemize}
	\item[(Q8)] For every $i \in [k-1]$ and $P \in \mc P_i$, $d(P, \mc A_i) \le 2 \omega$.
\end{itemize}
We explain the process in more detail below. \\

\noindent \emph{Removing a comparable pair}: Suppose $A\subset B$ is a comparable pair in $\mc A_i$ for some $i\in [k-1]$.

{
\narrower
\paragraph{Case 1:} $A$ and $B$ receive the same colour in at least a third of the colourings in $\fr C$.

Fix the more common monochromatic colouring of $A$ and $B$, and restrict $\fr C$ to those colourings.  Remove the sets $A$ and $B$ to $\mc R$.  We then run the branching up and down operation from $A \subset B$, branching up from $B$ to $\mc A_{i-1}$ and beyond, and branching down from $A$ to $\mc A_{i+1}$ and below, with the qualities (Q3) and (Q5) ensuring this can be carried out.

\paragraph{Case 2:} $A$ and $B$ receive different colours in at least two-thirds of the colourings in $\fr C$.

Colour $A$ and $B$ by the more common of the two patterns where they are oppositely coloured.  We restrict $\fr C$ to those colourings admitting this pattern on $A$ and $B$, and we remove the pair $A \subset B$ from $\mc A_i$ and place it in $\mc P_i$. \\

}

In this process, we have removed sets from $\mc A$, which could affect the qualities (Q3) and (Q5).  As in Stage IIb, we restore these properties by moving any violating sets to $\mc U$ or $\mc D$. \\

\noindent \emph{Moving to $\mc U \cup \mc D$}:  This operation is exactly as in Stage IIb. \\

We repeat this combination of operations --- first removing a comparable pair, and then moving sets to $\mc U \cup \mc D$ --- until there are no comparable pairs remaining in the parts $\mc A_i$, at which point Stage III ends.  The following lemma summarises the effects of this stage.

\begin{lemma} \label{lem:end-III}
At the end of Stage III, the following statements are true.
\begin{itemize}
	\item[(i)] The partition $\mc F = \cup_{i \in [k-1]} \left( \mc A_i \cup \mc U_i \cup \mc D_i \cup \mc P_i \right) \cup \mc R$ has the qualities (Q1)--(Q8).
	\item[(ii)] For every $F \in \mc U \cup \mc D \cup \mc P$, $\card{N_{\textup{pr}}(F)} \le 3 \omega$.  Moreover, if $\mc B_i \in \left\{ \mc U_i \cup \mc D_i \cup \mc P_i, \mc U_i \cup \mc D_{i-1}, \mc D_i \cup \mc U_{i+1} \right\}$ and $F \in \mc B_i$, then the left-neighbourhood of $F$ with respect to the right-to-left order of the members of $\mc A_i \cup \mc B_i$ is a subfamily of $N_{\textup{pr}}(F)$.
	\item[(iii)] If $t$ sets are moved to $\mc R$ during this stage, the family $\fr C$ of colourings shrinks by a factor of at most $2^{\eps t} 3^{\frac12 \card{\mc P}}$.
\end{itemize}
\end{lemma}

\begin{proof}
By Lemma~\ref{lem:end-II}(i), at the start of Stage III the partition has the qualities (Q1)--(Q6).  Since no new sets are introduced to any part $\mc A_i$ in Stage III, (Q1), (Q2) and (Q4) are unaffected, and continue to hold.  The moving to $\mc U \cup \mc D$ operations in Stage III ensure that (Q3) and (Q5) also hold by the end of this stage.  Since degrees into $\mc A$ can only decrease in Stage III, (Q6) holds for the previous members of $\mc U \cup \mc D$.  The proof that it holds for the new members is exactly as in Lemma~\ref{lem:end-II}(i).

For the new qualities, observe that (Q7) must hold at the end of Stage III, since if some $\mc A_i$ was not an antichain, then it would have a comparable pair that we could remove, and Stage III would not have ended.  To establish (Q8), observe that just before a set $P$ is moved from $\mc A_i$ into $\mc P_i$, it satisfies (Q2) and (Q4).  Hence $d(P, \mc A_i) = d^+(P, \mc A_i) + d^-(P, \mc A_i) \le 2 \omega$, and this degree can only decrease as the stage progresses.  Hence (Q8) also holds, thus showing (i) to be true.

We next consider (ii).  For the first statement, if $F \in \mc U \cup \mc D$, the proof from Lemma~\ref{lem:end-II}(ii) applies.  If $F \in \mc P_i$, then $N_{\textup{pr}}(F)$ is the set of neighbours of $F$ in $\mc A_i$ at the time $F$ was moved from $\mc A_i$ to $\mc P_i$.  As we have just shown above, there are at most $2 \omega$ such neighbours, and so the desired bound holds.  As for the second part of the statement, the only case that differs from Lemma~\ref{lem:end-II}(ii) is when $\mc B_i = \mc U_i \cup \mc D_i \cup \mc P_i$.  Here we observe that in the right-to-left order of $\mc A_i \cup \mc B_i$, the sets to the left of $F$ were all in $\mc A_i$ when $F$ was removed from $\mc A_i$.  Hence any left-neighbours belong to $N_{\textup{pr}}(F)$, as required.

Finally, we prove (iii).  Whenever we are in Case 2, and move a pair of sets to $\mc P$, $\fr C$ shrinks by a factor of at most three, while $\mc P$ gains two sets.  Hence these cases cause $\fr C$ to shrink by a factor of at most $3^{\frac12 \card{\mc P}}$.  In Case 1, the choice of the colouring of the pair $A \subset B$ causes $\fr C$ to shrink by a factor of at most six.  Lemma~\ref{lem:branch-colour} the controls the subsequent shrinkage caused by the branching operation, and it follows that if a total of $t$ sets are moved to $\mc R$, then $\fr C$ shrinks by a further factor of at most $2^{\eps t}$.
\end{proof}

\paragraph{Stage IV.} We are now very close to the partition promised in Proposition~\ref{prop:kchainstructure}.  All that remains is to partition $\mc R$ into the subparts $\mc R_i$, $i \in [k-1]$, that satisfy the following quality.
\begin{itemize}
	\item[(Q9)] For every $i \in [k-1]$ and $R \in \mc R_i$, $d(R, \mc A_i) \le 2 \omega$.
\end{itemize}
We achieve this with a sequence of branching operations, as outlined below. \\

\noindent \emph{Branching from  set in $\mc R$}: Suppose there is a set $F \in \mc R$ with more than $2 \omega$ neighbours in each part $\mc A_i$.  There are three possible cases for our branching operations.\footnote{Note that as $F \in \mc R$, it has the same colour in all of our colourings in $\fr C$.  Hence when we branch from $F$, we will build monochromatic chains in this common colour.}

{
\narrower
\paragraph{Case 1.}  $d^-(F, \mc A_1) > \omega$.

In this case we branch down from $F$, first to its subsets in $\mc A_1$, and then to lower parts, using (Q5).

\paragraph{Case 2.} $d^+(F, \mc A_{k-1}) > \omega$.

In this case we branch up from $F$, first to its supersets in $\mc A_{k-1}$, and then to higher parts, using (Q3).

\paragraph{Case 3.} $d^-(F, \mc A_1) \le \omega$ and $d^+(F, \mc A_{k-1}) \le \omega$.

Since $d^-(F, \mc A_1) \le \omega$ but $d(F, \mc A_1) > 2 \omega$, we must have $d^+(F, \mc A_1) > \omega$.  Let $i_0 \ge 1$ be the maximum index for which $d^+(F, \mc A_{i_0}) > \omega$, noting that $i_0 \le k-2$ as $d^+(F, \mc A_{k-1}) \le \omega$.  We must then have $d^+(F, \mc A_{i_0+1}) \le \omega$ and, since $d(F, \mc A_{i_0 + 1}) > 2 \omega$, we know $d^-(F, \mc A_{i_0 + 1}) > \omega$.  We can then branch up and down from $F$, using (Q3) to go from $F$ to its supersets in $\mc A_{i_0}$ and then to higher parts, and using (Q5) to go from $F$ to its subsets in $\mc A_{i_0 + 1}$ and on to lower parts.   \\

}

Again, colouring and removing sets from $\mc A$ to $\mc R$ could affect (Q3) and (Q5), so we move sets to $\mc U \cup \mc D$ to restore those qualities. \\

\noindent \emph{Moving to $\mc U \cup \mc D$}: This process is exactly as in Stage IIb. \\

We repeat this sequence of operations until there are no sets in $\mc R$ satisfying the assumption of the branching from $\mc R$ operation outlined above, at which point Stage IV is complete.  Our final lemma shows that Stage IV gives the desired partition of $\mc R$.

\begin{lemma} \label{lem:end-IV}
At the end of Stage IV, the following statements are true.
\begin{itemize}
	\item[(i)] $\mc R$ admits a partition $\mc R = \cup_{i \in [k-1]} \mc R_i$ such that $\mc F = \cup_{i \in [k-1]} \left( \mc A_i \cup \mc U_i \cup \mc D_i \cup \mc P_i \cup \mc R_i \right)$ has the qualities (Q1)--(Q9).
	\item[(ii)] For every $F \in \mc U \cup \mc D \cup \mc P$, $\card{N_{\textup{pr}}(F)} \le 3 \omega$.  Moreover, if $\mc B_i \in \left\{ \mc U_i \cup \mc D_i \cup \mc P_i, \mc U_i \cup \mc D_{i-1}, \mc D_i \cup \mc U_{i+1} \right\}$ and $F \in \mc B_i$, then the left-neighbourhood of $F$ with respect to the right-to-left order of the members of $\mc A_i \cup \mc B_i$ is a subfamily of $N_{\textup{pr}}(F)$.
	\item[(iii)] If $t$ sets are moved to $\mc R$ during this stage, the family $\fr C$ of colourings shrinks by a factor of at most $2^{\eps t}$.
\end{itemize}
\end{lemma}

\begin{proof}
Given $F \in \mc R$, let $i(F) = \min \{ i : d(F, \mc A_i) \le 2 \omega \}$.  Note that $i(F) \in [k-1]$ is well-defined, as if there was no such part $\mc A_i$ for a set $F \in \mc R$, Stage IV would not have ended.  This leads to the partition $\mc R = \cup_{i \in [k-1]} \mc R_i$, where $\mc R_i = \{ F \in \mc R: i(F) = i \}$, that satisfies (Q9).  As for qualities (Q1)--(Q8), note that they hold at the beginning of Stage IV by Lemma~\ref{lem:end-III}(i).  The proof that they remain valid is exactly as in Lemmas~\ref{lem:end-II}(i) and~\ref{lem:end-III}(i), which we need not repeat here.  This establishes (i).

Similarly, the proof that (ii) continues to hold with any new sets that might have been added to $\mc U \cup \mc D$ is just as in Lemmas~\ref{lem:end-II}(ii) and~\ref{lem:end-III}(ii), and hence (ii) is also true at the end of Stage IV.

Finally, (iii) follows directly from Lemma~\ref{lem:branch-colour}.
\end{proof}

At the end of Stage IV, we have partitioned our family $\mc F$ into parts $\mc A$, $\mc U$, $\mc D$, $\mc D$ and $\mc R$, which each part admitting a subpartition into $k-1$ further parts.  We close this subsection by proving that this partition has the five properties we required.

\begin{proof}[Proof of Proposition~\ref{prop:kchainstructure}]
Let $\mc F \subseteq 2^{[n]}$ be a family with at least one $(2,k)$-colouring.  As we have previously noted, $\mc F$ must be $(2k-1)$-chain-free, and therefore admits a partition into $2k-2$ antichains.  For any subfamily $\mc H \subseteq \mc F$, applying the pigeonhole principle to the intersection of $\mc H$ with these antichains shows that there is some antichain $\mc H' \subseteq \mc H$ with $\card{\mc H'} \ge \card{\mc H} / (2k-2)$, thus establishing (P5).

We now run $\mc F$ through Stages I to IV, by the end of which we have a partition
\[ \mc F = \cup_{i\in[k-1]} \left( \mc A_i \cup \mc U_i \cup \mc D_i \cup \mc P_i \cup \mc R_i \right). \]

By Lemma~\ref{lem:end-IV}(i), this partition has the quality (Q7), which is that each $\mc A_i$ is an antichain.  Hence (P2) is satisfied.

Quality (Q6) asserts that $\mc U_1 = \mc D_{k-1} = \emptyset$, and that any set in $\mc U_i$, $\mc U_{i+1}$, $\mc D_i$ or $\mc D_{i-1}$ is comparable to at most $2 \omega$ sets in $\mc A_i$.  Qualities (Q8) and (Q9) establish the same bound for sets in $\mc P_i$ and $\mc R_i$ respectively.  Thus (P3) holds.

Now let $\mc B_i \in \left\{ \mc U_i \cup \mc D_i \cup \mc P_i, \mc U_i \cup \mc D_{i-1}, \mc D_i \cup \mc U_{i+1} \right\}$, and consider the right-to-left order of $\mc A_i \cup \mc B_i$.  The number of comparable pairs in $\mc A_i \cup \mc B_i$ is simply the sum of the left-degrees of the sets in this order, which we sum from right to left.  By Lemma~\ref{lem:end-IV}(ii), we know that for every set $F \in \mc B_i$, its left-neighbourhood is a subfamily of $N_{\textup{pr}}(F)$, which has size at most $3 \omega$.  Hence we obtain at most $3 \omega \card{\mc B_i}$ comparable pairs involving sets in $\mc B_i$.  This leaves $\mc A_i$, which, by (Q7), is an antichain, and thus has no comparable pairs.  Hence $\mc A_i \cup \mc B_i$ has at most $3 \omega \card{\mc B_i}$ comparable pairs, giving (P4).

Finally, we bound the number of $(2,k)$-colourings.  We considered all $(2,k)$-colourings of $\mc F$, but by the end of Stage IV restricted our attention to a subset $\fr C$ of colourings where the sets in $\mc P$ and $\mc R$ had already been coloured.  Combining Lemmas~\ref{lem:end-I},~\ref{lem:end-II}(iii),~\ref{lem:end-III}(iii) and~\ref{lem:end-IV}(iv), $\fr C$ shrinks by a factor of at most $2^{\eps \card{\mc R}} 3^{\frac12 \card{\mc P}}$ throughout the process.  Moreover, since only sets in $\mc A \cup \mc  U \cup \mc D$ remain to be coloured, we must have $\card{\fr C} \le 2^{\card{\mc A} + \card{\mc U} + \card{\mc D}}$.  This shows that the number of $(2,k)$-colourings of $\mc F$ is in total at most $2^{\card{\mc A} + \card{\mc U} + \card{\mc D} + \eps \card{\mc R}} 3^{\frac12 \card{\mc P}}$, as required for (P1).

Hence the partition provided by the procedure has all the desired properties.
\end{proof}

\subsection{Proofs of the lemmata} \label{subsec:supersat}

Now that we have seen how a family with $(2,k)$-colourings can be partitioned, and how that partition can be used to bound the number of such colourings, all that remains to complete the proof of Theorem~\ref{thm:2-colours} is to prove the lemmata from Section~\ref{subsec:counting}.  Recall that we defined the weight of a set $F \subseteq [n]$ as $w_k(F) = \min \left\{ \binom{n}{\card{F}}^{-1}, \binom{n}{\floor{\frac{n-k}{2}}}^{-1} \right\}$.  We use the following bound, valid for $n \ge 4k^2$.
\begin{equation}\label{weight-asym}
\binom{n}{\floor{n/2}}^{-1}\le w_k(F) \le \binom{n}{\floor{\frac{n-k}{2}}}^{-1} \le \left(1+\frac{2k^2}{n}\right) \binom{n}{\floor{n/2}}^{-1}.
\end{equation}

We first prove Lemma~\ref{lem:transference}, which allows us to convert between family weights and sizes.

\begin{proof}[Proof of Lemma \ref{lem:transference}]
Let $\mc H$ be the lightest subfamily of $2^{[n]}$ containing $m_{k-1}+t$ sets.  Since the weights of sets increase as we move away from the middle level, we can assume that $\mc H$ contains the $m_{k-1}$ sets from the $k-1$ middle levels, and $t$ additional sets from the next level.  The $k-1$ middle levels have total weight $k-1$, and by~\eqref{weight-asym} each of the additional sets have weight at least $\binom{n}{\floor{n/2}}^{-1}$, and thus
\begin{equation}\label{weight-H-lower}
\weight(\mc H) \ge k-1+ t \binom{n}{\floor{n/2}}^{-1}.
\end{equation}

If $\card{\mc F_0} \ge m_{k-1}+t$, then by the choice of $\mc H$, one has $\weight(\mc F_0)\ge \weight(\mc H) \ge k-1+ t \binom{n}{\floor{n/2}}^{-1}$. Since $\alpha_i$ and $\weight(\mc F_i)$ are non-negative for every $1\le i \le s$, this implies the desired inequality, $\weight(\mc F_0)+\left(1+\tfrac{2k^2}{n}\right) \sum_{i=1}^{s}\alpha_i\weight(\mc F_i) \ge k-1+t \binom{n}{\floor{n/2}}^{-1}$.

Hence we suppose $\card{\mc F_0}<m_{k-1}+t$. Let us denote by $\mc H_0$ a lightest subfamily of $2^{[n]}$ of size $\card{\mc F_0}$, and note that we may take $\mc H_0\subset\mc H$.  Then $\weight(\mc H_0) \le \weight(\mc F_0)$ and $\card{\mc H\setminus\mc H_0}=m_{k-1}+t-\card{\mc F_0}\le \sum_{i=1}^{s}\al_i\card{\mc F_i}$.  Using~\eqref{weight-asym} in the first and third inequalities below, it follows that
\begin{align}\label{weight-H-upper}
\notag \weight(\mc H)=\weight(\mc H_0)+\weight(\mc H\setminus\mc H_0) &\le \weight(\mc F_0)+\left(1+\frac{2k^2}{n}\right)\card{\mc H\setminus\mc H_0}\binom{n}{\floor{n/2}}^{-1}\\
\notag & \le \weight(\mc F_0)+\left(1+\frac{2k^2}{n}\right)\sum_{i=1}^{s} \alpha_i\card{\mc F_i}\binom{n}{\floor{n/2}}^{-1}\\
&\le \weight(\mc F_0)+\left(1+\frac{2k^2}{n}\right)\sum_{i=1}^{s}\alpha_i\weight(\mc F_i).
\end{align}

Combining \eqref{weight-H-lower} and \eqref{weight-H-upper} now gives the desired inequality.
\end{proof}

We next turn to the supersaturation result of Lemma~\ref{lem:2-supersat}, which states that families of large weight must contain many comparable pairs.  Our proof of this result requires the following variant of the LYM inequality~\cite{Lub66,Mes63,Yam54}.

\begin{lemma} \label{lem:LYM}
If $\mc F$ is a subfamily of $2^{[n]}$ that does not contain $[n]$, then
\begin{equation}\label{ineq:LYMbound}
\sum_{F \in \mc F} \binom{n}{\card{F}}^{-1}\lt(1-\frac{d^{+}(F,\mc F)}{n-\card{F}}\rt)\le 1.
\end{equation}	
\end{lemma}

\begin{proof}
Given a permutation $\sigma$ of $[n]$, let $m_1(\sigma)$ denote the number of sets from $\mc F$ appearing as a prefix in $\sigma$, and let $m_2(\sigma) = \binom{m_1(\sigma)}{2}$ be the number of pairs of such sets.  Since $m - \binom{m}{2} \le 1$ for every natural number $m$, we have $m_1(\sigma) - m_2(\sigma) \le 1$.  Summing over every permutation $\sigma$ and double-counting, we have
\[ \sum_{F \in \mc F} \card{F}! (n - \card{F})! - \sum_{\substack{F,G\in \mc F \\ F\subset G}} \card{F}!(\card{G} - \card{F})!(n-\card{G})! = \sum_{\sigma \in S_n} m_1(\sigma) - \sum_{\sigma \in S_n} m_2(\sigma) \le n!. \]
Given $F$, the quantity $(\card{G} - \card{F})!(n-\card{G})!$ is maximised for $F \subsetneq G \subsetneq [n]$ when $\card{G}-\card{F}=1$.  Hence
\[ \sum_{\substack{F, G \in \mc F \\ F \subset G}} \card{F}! (\card{G} - \card{F})! (n - \card{G})! = \sum_{F \in \mc F} \card{F}! \sum_{\substack{G \in \mc F \\ F \subset G}} (\card{G} - \card{F})!(n - \card{G})! \le \sum_{F \in \mc F} d^+(F, \mc F) \card{F}! (n - \card{F} - 1)!. \]
Making this substitution and dividing through by $n!$ gives~\eqref{ineq:LYMbound}.
\end{proof}

Using this result, we can prove Lemma~\ref{lem:2-supersat}.  In what follows, we denote by $\textup{cp}(\mc F)$ the number of comparable pairs in a set family $\mc F$.

\begin{proof}[Proof of Lemma \ref{lem:2-supersat}]
We prove the lemma by induction on $\card{\mc F}$. Note that the statement holds vacuously for $\mc F=\emptyset$.

We now proceed to the induction step with $\mc F \ne \emptyset$. The statement is trivial for $r\le 0$, since the number of comparable pairs in $\mc F$ is non-negative. Now consider the case $r>0$. Note that, by \eqref{weight-asym}, since $n\ge 2\delta^{-1}k^2$ and $F\subset [n]$, we have
\begin{equation}\label{weight-asym-2}
\binom{n}{\floor{n/2}}^{-1} \le \weight(F)\le (1+\delta)\binom{n}{\floor{n/2}}^{-1}.
\end{equation}

In what follows we shall reduce the problem to the case when 
\begin{equation}\label{supersat:bound_degrees}
\max\{d^{+}(F,\mc F),d^{-}(F,\mc F)\}<\left(\tfrac12-\tfrac{\de}{2}\right)n \ \text{for all $F\in\mc F$}.
\end{equation} 
Indeed, if this bound does not hold, then there is some set $F$ that is involved in at least $(\tfrac12 - \tfrac{\delta}{2})n$ comparable pairs.  Applying~\eqref{weight-asym-2} gives
\[
\weight(\mc F\setminus\{F\}) = \weight(\mc F)-\weight(F) \ge 1+(r-1-\delta)\binom{n}{\floor{n/2}}^{-1}.
\]
The induction hypothesis thus implies $\textup{cp}(\mc F\setminus\{F\}) \ge \left(\tfrac{1}{2}-\de\right)(r-1-\de)n$. Adding the comparable pairs involving $F$, $\textup{cp}(\mc F) \ge \left(\tfrac{1}{2}-\de\right)(r-1-\de)n+\left(\tfrac12-\tfrac{\de}{2}\right)n>\left(\tfrac{1}{2}-\de\right)rn$, as desired.
	
We remark that the degree condition \eqref{supersat:bound_degrees} implies $[n] \notin \mc F$. Indeed, by \eqref{weight-asym-2}, one has 
\[
1\le 1 + r \binom{n}{\floor{n/2}}^{-1} \le \weight(\mc F) \le (1+\delta) \card{\mc F} \binom{n}{\floor{n/2}}^{-1}.
\]
As $\delta \in (0, \frac12)$, we must have $\card{\mc F} \ge \frac12\binom{n}{\floor{n/2}}$, and thus $d^{-}([n],\mc F) \ge \card{\mc F}-1>n$ for $n$ sufficiently large. By~\eqref{supersat:bound_degrees}, it follows that $[n] \notin \mc F$.

Let $\mc L \subseteq \mc F$ consist of those sets whose sizes are at least $\left(\tfrac{1}{2}+\tfrac{\delta}{2}\right)n$, let $\ell = \card{\mc L}$, and let $\mc L_0 \subseteq \mc L$ be the subfamily of inclusion-minimal sets within $\mc L$.  Since $d^+(F, \mc F) < n$ for all $F \in \mc F$ by~\eqref{supersat:bound_degrees}, we must have
\begin{equation}\label{supersat:min_sets}
\card{\mc L_0}\ge \frac{\ell}{n}.
\end{equation}

We write $\mc S = \mc F \setminus \mc L$ for the subfamily of sets that have size smaller than $\left( \tfrac12 + \tfrac{\delta}{2} \right)n$.  By~\eqref{weight-asym-2},
\begin{equation}\label{supersat:weight}
\weight(\mc S)\ge \weight(\mc F)-(1+\delta) \ell \binom{n}{\floor{n/2}}^{-1}\ge 1+(r-(1+\delta)\ell)\binom{n}{\floor{n/2}}^{-1}.
\end{equation}

We now apply Lemma~\ref{lem:LYM} to the family $\mc S \cup \mc L_0$, obtaining
\begin{equation}\label{supersat:LYMtotal}
1 \ge \sum_{F \in \mc S \cup \mc L_0} \binom{n}{\card{F}}^{-1} \left( 1 - \frac{d^+(F, \mc S \cup \mc L_0)}{n - \card{F}} \right).
\end{equation}

We split this sum based on whether the sets are in $\mc L_0$ or $\mc S$.  For the former, note that $d^+(F, \mc S \cup \mc L_0) = 0$, since these sets are too large to be contained in any sets from $\mc S$, and $\mc L_0$ is an antichain.  As $\binom{n}{\card{F}} \le n^{-2}\binom{n}{\floor{n/2}}$ when $\card{F} \ge (\tfrac12 + \tfrac{\delta}{2})n$ and $n \ge C \delta^{-3}$ for large enough $C$, this gives
\begin{equation} \label{supersat:LYMantichain}
\sum_{F \in \mc L_0} \binom{n}{\card{F}}^{-1} \left(1 - \frac{d^+(F, \mc S \cup \mc L_0)}{n - \card{F}} \right) = \sum_{F \in \mc L_0} \binom{n}{\card{F}}^{-1} \ge \card{\mc L_0} n^2 \binom{n}{\floor{n/2}}^{-1} \ge \ell n \binom{n}{\floor{n/2}}^{-1},
\end{equation}
where we use~\eqref{supersat:min_sets} in the final inequality.

On the other hand, for $F \in \mc S$, observe that~\eqref{supersat:bound_degrees} implies $d^+(F, \mc F) \le (\tfrac12 - \tfrac{\delta}{2})n \le n - \card{F}$.  Hence, since $w_k(F) \le \binom{n}{\card{F}}^{-1}$, 
\[ \sum_{F \in \mc S} \binom{n}{\card{F}}^{-1} \left( 1 - \frac{d^+(F, \mc S \cup \mc L_0)}{n - \card{F}} \right) \ge \sum_{F \in \mc S} w_k(F) \left( 1- \frac{d^+(F, \mc F)}{(\tfrac12 - \tfrac{\delta}{2})n} \right) = w_k(\mc S) - \sum_{F \in \mc S} \frac{w_k(F) d^+(F, \mc F)}{(\tfrac12 - \tfrac{\delta}{2})n}.\]
Now observe that~\eqref{weight-asym-2} gives a uniform upper bound on $w_k(F)$, while $\sum_{F \in \mc S} d^+(F, \mc F) \le \textup{cp}(\mc F)$.  Moreover,~\eqref{supersat:weight} provides a lower bound for $w_k(\mc S)$.  Combining this with~\eqref{supersat:LYMtotal} and~\eqref{supersat:LYMantichain}, we have
\[ 1 \ge 1 + \left[ (r - (1 + \delta) \ell) - \frac{(1 + \delta) \textup{cp}(\mc F)}{( \tfrac12 - \tfrac{\delta}{2} ) n } + \ell n \right] \binom{n}{\floor{n/2}}^{-1}. \]

Solving for $\textup{cp}(\mc F)$ gives
\[ \textup{cp}(\mc F) \ge \frac{\tfrac12 - \tfrac{\delta}{2}}{1 + \delta} \left( r - (1 + \delta) \ell + \ell n \right)n \ge \frac{ \tfrac12 - \tfrac{\delta}{2} }{1 + \delta} rn \ge (\tfrac12 - \delta) rn, \]
where we use $r > 0$ in the final inequality.  This completes the proof of Lemma~\ref{lem:2-supersat}.
\end{proof}

\section{Asymptotics via containers}\label{sec:logasym}

In this section we shall prove Proposition~\ref{prop:logasym}, obtaining general upper bounds on $f(r,k;n)$ that are log-asymptotically tight whenever $r(k-1)$ is divisible by three.  
We make use of the theory of hypergraph containers developed by Balogh, Morris and Samotij~\cite{BMS15} and Saxton and Thomason~\cite{ST15}. In essence, what this theory says is that if the edges of a uniform hypergraph $\mathcal{H}$ are fairly evenly distributed, then there is a relatively small collection of `containers', each not too large, which cover the family of independent sets of $\mathcal{H}$.
We shall in particular apply the following result of Collares Neto and Morris (Theorem 4.2 in~\cite{CM15}) concerning containers for $k$-chain-free families in $2^{[n]}$.

\begin{prop}[Collares Neto--Morris~\cite{CM15}] \label{prop:containers}
For every $k \ge 2$, $\eps' > 0$ and $n$ sufficiently large, there is a set of containers $\Gamma \subseteq 2^{2^{[n]}}$ such that:
\begin{itemize}
	\item[(a)] each container $\mc C \in \Gamma$ has size $|\mc C| \le (k-1 + \eps') \binom{n}{\floor{n/2}}$,
	\item[(b)] every $k$-chain-free family $\mc I \subseteq 2^{[n]}$ is contained in some container $\mc C(\mc I) \in \Gamma$, and
	\item[(c)] the number of containers, $\card{\Gamma}$, is bounded from above by $\exp \left( \eps' \binom{n}{\floor{n/2}} \right)$.
\end{itemize}
\end{prop}

We now restate our asymptotic result before presenting its short proof.

\logasym*

\begin{proof}[Proof of Proposition \ref{prop:logasym}]
Let $\mc F$ be any set family over $[n]$, and let $c(\mc F)$ denote the number of $(r,k)$-colourings of $\mc F$.  Further, let $\Gamma$ be the set of containers given by Proposition~\ref{prop:containers} for the parameters $k$ and $\eps' = \eps/4$. We wish to show that $c(\mc F)\le 3^{\frac13 r (k-1 + 4 \eps') \binom{n}{\floor{n/2}}}$.

Observe that the colour classes of every $(r,k)$-colouring of $\mc F$ give a partition $\mc F=\mc I_1\cup \ldots \cup \mc I_r$ into $k$-chain-free families. We can then map the $(r,k)$-colourings of $\mc F$ to $r$-tuples of containers $(\mc C_1,\ldots,\mc C_r) \in \Gamma^r$, where $\mc C_i=\mc C(\mc I_i)$ for $1\le i\le r$. Let $c(\mc C_1,\ldots,\mc C_r)$ denote the number of $(r,k)$-colourings of $\mc F$ mapped to the $r$-tuple $(\mc C_1,\ldots,\mc C_r)$. By the pigeonhole principle, there exists an $r$-tuple $(\mc C_1,\ldots,\mc C_r) \in \Gamma^r$ with $c(\mc C_1,\ldots,\mc C_r)\ge c(\mc F)\card{\Gamma}^{-r}$.  Fix such an $r$-tuple and consider the corresponding colourings.

Given $F \in \mc F$, we write $t(F)$ for the number of indices $i$ such that $F \in \mc C_i$.  We then have
\[ \sum_{F \in \mc F} t(F) \le \sum_{i=1}^r \card{\mc C_i} \le r (k-1 + \eps') \binom{n}{\floor{n/2}}. \]
Moreover, the set $F$ can be coloured by the colour $i$ only if $F \in \mc C_i$, and hence there are at most $t(F)$ colours available for $F$.  Thus $c(\mc C_1,\ldots,\mc C_r)$ is bounded by $\prod_{F \in \mc F} t(F)$.  Some straightforward optimisation (see, for instance, Lemma 3.1 in~\cite{CDT16}) shows that this expression is maximised subject to the upper bound on the sum when each $t(F)$ is equal to $3$, and so $c(\mc C_1,\ldots,\mc C_r)$ is at most $3^{ \frac13 r(k-1 + \eps') \binom{n}{\floor{n/2}}}$.
Hence, as required, 
\[ c(\mc F) \le c(\mc C_1,\ldots,\mc C_r)\card{\Gamma}^r \le 3^{ \frac13 r(k-1 + \eps') \binom{n}{\floor{n/2}}} \exp \left( r \eps' \binom{n}{\floor{n/2}} \right) \le 3^{ \frac13 r (k-1 + 4 \eps') \binom{n}{\floor{n/2}}}. \]

To see that this upper bound is essentially correct when $r(k-1)$ is divisible by three, let $\mc F$ be the $r(k-1) / 3$ largest levels of the Boolean lattice.  As $n$ is large enough in terms of $r,k$ and $\eps$, all of these levels are approximately the same size, and we have at least $\left(\frac13 r(k-1 - \eps) \right) \binom{n}{\floor{n/2}}$ sets in total.  

\begin{claim}
We can assign each colour to $k-1$ levels in such a way that each level is assigned three colours.
\end{claim}

\begin{proof}
For convenience we name the levels $\ell_1,\ell_2,\ldots,\ell_{r(k-1)/3}$. Let $L$ be the ordered list $(\ell_1,\ldots,\ell_{r(k-1)/3},\ell_1,\ldots,\ell_{r(k-1)/3},\ell_1,\ldots,\ell_{r(k-1)/3})$ in which each level appears three times in $L$. We then assign the first colour to the first $k-1$ levels in $L$, the second colour to the next $k-1$ levels in $L$, and so on.
Since there are $r(k-1)/3 \ge k-1$ distinct levels, no colour is assigned to the same level twice, so each colour is used on $k-1$ distinct levels. As each level appears three times in $L$, it gets three distinct colours.
\end{proof}

Using the assignment of colours given by the claim, we colour each set with one of its three available colours arbitrarily.  Since each colour class spans $k-1$ levels, there are no monochromatic $k$-chains, and hence each such colouring is an $(r,k)$-colouring of $\mc F$.  This shows that $f(r,k;n) \ge c(\mc F) \ge 3^{\frac13 r(k-1 - \eps) \binom{n}{\floor{n/2}}}$.
\end{proof}

When three does not divide $r(k-1)$, we do not have a construction matching our upper bound, and indeed, we do not believe it to be tight.  In the concluding remarks we make some suggestions as to what the truth might be.

\section{Concluding remarks}\label{sec:conclusion}

In this paper we initiated the study of the Erd\H{o}s--Rothschild problem in the context of Sperner theory, which asks for the set families over the ground set $[n]$ with the maximum number of $(r,k)$-colourings, which are $r$-colourings of the set family that avoid monochromatic $k$-chains.  We showed that for $(r,k) \in \{ (3,2) \} \cup \{ (2,k) : k \ge 2 \}$ (and $n$ sufficiently large), the optimal families for the Erd\H{o}s--Rothschild problem are the largest $k$-chain-free families, which are the $k-1$ middle levels of the Boolean lattice.

We further showed that these families need not be optimal for larger values of $r$, as larger set families, which contain many $k$-chains, may still admit more $(r,k)$-colourings.  For example, when $n$ is odd, the union of the two largest uniform levels, which are each maximum-sized antichains, contains more $(4,2)$-colourings than either of the antichains alone.  However, when $n$ is even, there is a unique largest antichain, and this might still maximise the number of $(4,2)$-colourings.

\begin{ques}
If $n$ is even and sufficiently large, does $\binom{[n]}{n/2}$ maximise the number of $(4,2)$-colourings?
\end{ques}

Just as with the $(2,k)$-colourings, we would also expect the number of $(3,k)$-colourings to be maximised by the largest $k$-chain-free families.  In our proof for the $(2,k)$ case, we used the fact that if one has a monochromatic $(k-1)$-chain $F_1 \subset F_2 \subset \hdots \subset F_{k-1}$, then the colour of any set containing $F_{k-1}$ is determined.  When one has three colours, however, the colours of such sets are merely restricted to being one of the two other colours.  While this is quite a severe restriction when there are many $k$-chains present, we could not exploit it in our calculations to deduce that the $k$-chain-free families were optimal.  It would appear that some further arguments may be necessary.

\begin{ques}
For $k \ge 3$ and $n$ sufficiently large, do the largest $k$-chain-free families in $2^{[n]}$ also maximise the number of $(3,k)$-colourings?
\end{ques}

In Section~\ref{sec:logasym}, we used containers for $k$-chain-free families to obtain an upper bound on the number of $(r,k)$-colourings a family could have, showing that for every $\eps > 0$ and $n$ sufficiently large, $f(r,k;n) \le 3^{\frac13 r(k-1 + \eps) \binom{n}{\floor{n/2}}}$.  When $r(k-1)$ is divisible by three, we can show this bound to be log-asymptotically sharp, as one can distribute $r$ colours over $\frac13 r(k-1)$ uniform levels in such a way that every set has three available colours, which matches the solution to the optimisation problem in the upper bound.

When $r(k-1)$ is not divisible by three, such a partition of the colours is not feasible.  If we could apply the optimisation problem to the uniform levels, instead of to the individual sets, then it would be best to take $\ceil {r(k-1)/ 3}$ levels, and distribute the $r$ colours in such a way that each colour is used on $k-1$ levels, and all levels have three colours, except for one or two that only receive two colours.  If this construction is indeed best possible, one would need to improve the upper bound to obtain log-asymptotically sharp results.

\begin{ques}
Can we improve the upper bound of Proposition~\ref{prop:logasym} when $r(k-1)$ is not divisible by three?  If we could show, for $\eps > 0$ and $n$ large enough,
\[ f(r,k;n) \le \begin{cases}
\left( 2^2 \cdot 3^{ \frac{r(k-1) - 4}{3} } \right)^{(1 + \eps) \binom{n}{\floor{n/2}}} & \mbox{if } r(k-1) \equiv 1 \mod 3, \\
\left( 2 \cdot 3^{ \frac{r(k-1) - 2}{3} } \right) ^{(1 + \eps) \binom{n}{\floor{n/2}}} & \mbox{if } r(k-1) \equiv 2 \mod 3,
\end{cases} \]
then we would have log-asymptotically correct bounds in all cases.
\end{ques}

Finally, from a more general viewpoint, one could broaden the study of Erd\H{o}s--Rothschild problems, and seek to extend various other extremal problems in this fashion.  In many of the problems studied to date, the families maximising the number of $2$- or $3$-Erd\H{o}s--Rothschild-colourings are those solving the original extremal problem.\footnote{Reference \cite{HKL14} gives instances for which this is not the case.}  It would be very interesting to develop a ``metatheorem", identifying which features of an extremal problem ensure that the trivial lower bound is tight for the two- and three-colour Erd\H{o}s--Rothschild problems.

\end{document}